\documentclass[10pt,twoside, a4paper, english, reqno]{amsart}
\usepackage[dvips]{epsfig}
\usepackage{amscd}
\usepackage{a4wide}
\usepackage{amssymb}
\usepackage{amsthm}
\usepackage{comment}
\usepackage{amsmath}
\usepackage{bm} 
\usepackage{latexsym}
\usepackage[a4paper,inner=2cm,outer=2cm,top=2.5cm,bottom=2.5cm]{geometry}

\usepackage{upref}
\usepackage{enumerate}
\usepackage{xcolor}
\usepackage{mathtools}
\usepackage[ocgcolorlinks, linkcolor=blue,citecolor=red,urlcolor=blue]{hyperref}
\usepackage[foot]{amsaddr}
\newcommand{\mb}{\mathbb}
\newcommand{\D}{\mathrm{d}}
\newcommand{\mc}{\mathcal}
\newcommand{\m}{\mbox}
\newcommand{\ld}{\lambda}
\newcommand{\og}{\omega}

\newcommand{\rt}{\sqrt}

\newcommand{\dd}{\partial}
\newcommand{\tb}{\textbf}

\newcommand{\nn}{\nonumber}

\newcommand{\al}{\alpha}

\newcommand{\iy}{\infty}

\newtheorem{theorem}{Theorem}[section]
\newtheorem{proposition}[theorem]{Proposition}
\newtheorem{lemma}[theorem]{Lemma}

\newtheorem{remark}[theorem]{Remark}
\newtheorem{corollary}[theorem]{Corollary}   
\numberwithin{figure}{section}
\numberwithin{table}{section}
\allowdisplaybreaks	
\newcommand{\R}{\mathbb{R}}

\newcommand{\HS}{\mathbf{H}}
\newcommand{\LS}{\mathbf{L}}
\newcommand{\CS}{\mathbf{C}}

\newcommand{\Lc}{\mathcal{L}}

\newcommand{\lb}{\left(}

\newcommand{\rb}{\right)}

\renewcommand{\O}{\Omega}
\renewcommand{\o}{\omega}
 \mathtoolsset{showonlyrefs}
  \def\diam{ {{\rm diam}}}
    \def\div{ {{\rm div}}}
\def\diag{ {\rm diag} }
\def\supp{ {\rm supp} }

\newcommand{\gs}{\overrightarrow{g}}
\newcommand{\hs}{\overrightarrow{h}}
\newcommand{\us}{\overrightarrow{u}}
\newcommand{\vs}{\overrightarrow{v}}

\newcommand{\Sb}{\mathbb{S}}
\numberwithin{equation}{section} \allowdisplaybreaks
\usepackage[normalem]{ulem}	
\usepackage{cancel}
\title[Stability estimate of matrix potential]{Stability in the determination of the coefficients for coupled relativistic Schr\"odinger system
}
\author[N. Kumar, P. Kumar and  T. Sarkar]
		{Nitesh Kumar$^\diamond$, 
        Parveen Kumar$^\dagger$
        and Tanmay Sarkar$^\ddagger$}	
		\date{}
		\address[]{$^\diamond$Department of Mathematics, 
			Indian Institute of Technology Jammu,
			NH-44 Bypass Road, Nagrota PO, Jagti, 
			Jammu - 181221, INDIA.\\
             {Department of Applied Science and Humanities, Government Engineering College Gopalganj, Dist.-Gopalganj, Bihar-841501, India.}}
		\email[]{2019rma0004@iitjammu.ac.in,\,{nitesh.maths.dstte@bihar.gov.in}}
        \address[]{$^{\dagger}$Department of Mathematics,  Indian Institute of Technology Ropar, Rupnagar, Punjab-140001, INDIA.}
		\email[]{\tt parveen.24maz0013@iitrpr.ac.in}
		\address[]{$^\ddagger$Department of Mathematics, 
			Indian Institute of Technology Jammu,
			NH-44 Bypass Road, Nagrota PO, Jagti, 
			Jammu - 181221, INDIA.}
		\email[]{tanmay.sarkar@iitjammu.ac.in}

\begin{document}
\begin{abstract} 
In this article, we prove a stability result for the identification of the coefficients appearing in a coupled relativistic Schr\"odinger system posed on a bounded domain of $\mathbb{R}^{1+n}$, where $n \geq 3$. More precisely, by utilizing input-output operator, where Neumann boundary data are available only on a subset of the boundary, and constructing an appropriate geometric optics solution, we derive a stability estimate of $\log$-$\log$ type for the matrix-valued first-order coefficients. This result is subsequently used to obtain a $\log$-$\log$-$\log$ type stability estimate for the zeroth-order matrix-valued potential.

\vspace{0.2 cm}

\vspace{0.2 cm}
\end{abstract}
\maketitle

\section{Introduction}
\subsection{Mathematical model} Let $\O\subset \mathbb{R}^n$ be a bounded domain with smooth boundary $\partial \O$, where $n\geq 3.$
For $T>\diam(\O)$, let us introduce the coupled relativistic Schr\"odinger operator $\mathcal{L}_{\mathcal{A},\mathcal{B},q}$ in $Q:=(0,T)\times \O$, given by
\begin{align}\label{Def: Govern operator}
\mathcal{L}_{\mathcal{A},\mathcal{B},q}\overrightarrow {u}(t,x) :
={\displaystyle  \begin{bmatrix} (\partial_t+A^1_0(t,x))^2u_1(t,x)-\sum\limits_{j=1}^n(\partial_j+A^1_j(t,x))^2u_1(t,x)+\sum\limits_{j=1}^{m}q_{1j}(t,x)u_j(t,x) \\  (\partial_t+A^2_0(t,x))^2u_2(t,x)-\sum\limits_{j=1}^n(\partial_j+A^2_j(t,x))^2u_2(t,x)+\sum\limits_{j=1}^{m} q_{2j}(t,x)u_j(t,x) \\ \vdots \\ (\partial_t+A^m_0(t,x))^2u_m(t,x)-\sum\limits_{j=1}^n(\partial_j+A^m_j(t,x))^2u_m(t,x)+\sum\limits_{j=1}^{m} q_{mj}(t,x)u_j(t,x) \end{bmatrix}},
\end{align}
where $\displaystyle \overrightarrow {u}(t,x):=\left(u_1(t,x), u_2(t,x),\cdots, u_m(t,x)\right)^{T}$ represent the displacement vector.
 Throughout this article, we use the following notation convention: for a given vector $\overrightarrow{v}:=(v_1,v_2,\cdots,v_m)$, we define
$\partial_t \overrightarrow{v}:=(\partial_t v_1, \partial_t v_2,\cdots,\partial_t v_m)^T$. Similarly, the operators $\nabla_x$ and $\Delta_x$ also act componentwise on the given vector; that is,  $\nabla_x \overrightarrow{v}:= (\nabla_x v_1, \nabla_x v_2,\cdots,\nabla_x v_m)^T 
$ and $\Delta_x \overrightarrow{v}:= (\Delta_x v_1, \Delta_x v_2,\cdots,\Delta_x v_m)^T$. Let $\square:= \partial_t^2-\Delta_{x}$ denote the standard wave operator and $\overrightarrow{A}^i:=(A^i_1,A^i_2,\cdots,A^i_n)$ for each $1\leq i\leq m$. Now, if we denote 
\begin{equation}\label{Def: A, B and q}
\begin{split}
\mathcal{A}(t,x)&:=\diag\left(A^{1}_0(t,x),\,A^{2}_0(t,x),\,\cdots,\,A^{m}_0(t,x)\right),\\
\mathcal{B}(t,x)&:=\diag\left(\overrightarrow{A}^{1}(t,x),\,\overrightarrow{A}^{2}(t,x),\,\cdots,\,\overrightarrow{A}^{m}(t,x)\right),\quad  \\q(t,x)&:=(q_{ij}(t,x))_{1\leq i,j\leq m},
\end{split}
\end{equation}
then the aforementioned operator can be compactly written as 
\begin{equation}\label{Govern operator: compact form}
\begin{split}
&\mathcal{L}_{\mathcal{A},\mathcal{B},q}\overrightarrow {u}(t,x)=\square \overrightarrow{u}(t,x) + \mathcal{A}(t,x) \,  \partial_t\overrightarrow{u}(t,x) -\mathcal{B} (t,x)\cdot \nabla_x\overrightarrow{u}(t,x)+\big(\widetilde{\mathcal{A}}(t,x)+q(t,x)\big)\overrightarrow{u}(t,x), 
\end{split}
\end{equation}
where
 \[
  \widetilde{\mathcal{A}}(t,x) :=\diag(\widetilde{A^1}(t,x),\,\widetilde{A^2}(t,x),\,\dots,\,\widetilde{A^m}(t,x))
\] with
\begin{align}\label{Def: Tilda A}
\widetilde{A^i}(t,x)=\partial_tA^i_0(t,x)+(A^i_0)^2(t,x)-\nabla_x\cdot\overrightarrow{A}^i(t,x)-|\overrightarrow{A}^i(t,x)|^2,\quad i=1,2,\dots,m.
\end{align}
Also, the dot product $\mathcal{B} \cdot \nabla_x\overrightarrow{u}$ is defined by
\begin{align}
  \mathcal{B} (t,x)\cdot \nabla_x\overrightarrow{u}(t,x):=\left(\overrightarrow{A}^1(t,x)\cdot\nabla_xu_1(t,x),\,\overrightarrow{A}^2(t,x)\cdot\nabla_x u_2(t,x),\,\cdots,\,\overrightarrow{A}^m(t,x)\cdot\nabla_x u_m(t,x)\right)^T .
  \end{align}
  Throughout the paper, boldface notation denotes vector-valued functions, each component of which belongs to the underlying function space. For example, $\mathbf{H}(X)$ denotes the space of vector-valued functions on a measurable set $X$ whose components belong to $H(X)$. 
  \par
  In the present paper, we consider a inverse related to stably determine the time-dependent coefficients appearing in the following initial boundary value problem (IBVP) for a linear coupled relativistic Schr\"odinger system
  \begin{equation}\label{Govern IBVP}
  \left\{\begin{array}{l c}
\mathcal{L}_{\mathcal{A},\mathcal{B},q}\overrightarrow {u}(t,x) = \overrightarrow {0}, &\quad (t,x)\in Q,\\
\overrightarrow {u}(0,x)=\overrightarrow {\psi_0}(x),~~ \partial_t \overrightarrow {u}(0,x) = \overrightarrow {\psi_1}(x), &\quad x\in\Omega,\\
\overrightarrow {u}(t,x) = \overrightarrow {f}(t,x), &\quad (t,x)\in  \Sigma,
\end{array}\right.
\end{equation}
where $\Sigma:=(0,T)\times \partial\O$ denotes the lateral boundary. For $1\leq i,j
\leq m$, let $A_0^i \in C_c^{\infty}(Q, \mathbb{R}), ~ \overrightarrow{A}^i \in  C_c^{\infty}(Q, \mathbb{R}^n)$ and  $q_{ij} \in L^{\infty}(Q)$. Then for $\overrightarrow{\psi_0}\in \mathbf H^1(\Omega),\,\overrightarrow{\psi_1}\in \mathbf L^2(\Omega)$ and $\overrightarrow{f}\in \mathbf H^1(\Sigma)$ such that the compatibility condition $\overrightarrow{\psi_0}|_{\partial\Omega}=\overrightarrow{f}(0,\cdot)|_{\partial\Omega}$ holds, the IBVP \eqref{Govern IBVP} admits a unique solution $\overrightarrow{u}\in  C^1\big([0,T];\mathbf L^2(\Omega)\big)\cap  C\big([0,T];\mathbf H^1(\Omega)\big)$(see Theorem \ref{Theorem:well_posedness}). Moreover, there exist a constant $C=C(\mc{A},\mc{B},\Omega,T)>0$ such that
\[\|\partial_\nu \overrightarrow{u}\|_{\mathbf L^2(\Sigma)}+\|\partial_t \overrightarrow{u}(t,\cdot)\|_{\mathbf L^2(\Omega)}+\|\overrightarrow{u}(t,\cdot)\|_{\mathbf H^1(\Omega)}\leq C\left(\|\overrightarrow{\psi_0}\|_{\mathbf H^1(\Omega)}+\|\overrightarrow{\psi_1}\|_{\mathbf L^2(\Omega)}+\|\overrightarrow{f}\|_{\mathbf L^2(\Sigma)}\right)\]
holds for any $t \in [0,T].$ The coupled system has various applications in engineering sciences. Related systems have applications involving wave propagation in power transmission networks \cite{JacekBanasiak}, studying output regulation and asymptotic tracking of prescribed reference signals \cite{IRSCHEID2023110748}, and characterizing dynamical boundary data in inverse problems for wave equations with matrix-valued potentials \cite{belishev2020datacharacterizationdynamicalinverse}. Also, a system of coupled Klein–Gordon equations is used to describe interacting scalar fields and provide effective continuum models for the long-wavelength dynamics of two-component media, including coupled periodic chains and elastic bi-layer materials \cite{KHUSNUTDINOVA20031}. This motivates us to investigate coupled wave equations with first-order perturbations and matrix-valued potentials.
\subsection{Inverse problem}
The main aim of this article is to derive a stable estimate of the coefficients from the input-output operator, with Neumann data measured on a suitable part of the boundary.  
Using the well-posedness of the forward problem, we define the input-output operator $\Lambda^*$ as 
\begin{align}
\begin{aligned}
    \Lambda^*:{\mathbf H^1(\Omega)\times \mathbf L^2(\Omega)\times \mathbf H^1(\Sigma)}&\mapsto\ { \mathbf H^1(\Omega)\times \mathbf L^2(\Sigma)} \\
\Lambda^*(\overrightarrow{\psi_0},\overrightarrow{\psi_1},\overrightarrow{f})&:=\left(\us(T,\cdot),\,\dd_\nu \us|_{\Sigma}\right),
\end{aligned}
\end{align}
where $\overrightarrow{u}$ is the solution of IBVP \eqref{Govern IBVP} and $\nu$ represents the  outward unit normal vector on $\partial \Omega$.
However, in the above-mentioned input-output operator, the Neumann data is prescribed on the whole boundary. 
To make sense of a partial input-output operator, we start by defining the illuminated and shadowed faces of the boundary $\partial \O$ with respect to some fixed $\omega_{0}\in\mathbb{S}^{n-1}$. Let
\begin{align}
\partial\Omega_{+,\omega_{0}}:=\left\{x\in\partial\Omega:\ \nu(x)\cdot\omega_{0}> 0 \right\}\quad\m{and}\quad \partial\Omega_{-,\omega_{0}}:=\left\{x\in\partial\Omega:\ \nu(x)\cdot\omega_{0}< 0 \right\}
\end{align}
denotes the  $\omega_0$-shadowed and $\omega_0$-illuminated face, respectively. Also, for any $\varepsilon>0$, we define
\begin{align}
  \partial\Omega_{+,\varepsilon,\og_0}&:=\left\{x\in\partial\Omega:\ \nu(x)\cdot\omega_{0}> \varepsilon  \right\}, \quad \partial\Omega_{-,\varepsilon,\og_0}:=\left\{x\in\partial\Omega:\ \nu(x)\cdot\omega_{0}<  \varepsilon \right\},\\   
  \Sigma_{\pm,\omega_0}&:=(0,T)\times\partial\Omega_{\pm,\omega_0}, \quad \Sigma_{\pm,\varepsilon,\omega_0}:=(0,T)\times \partial\Omega_{\pm,\varepsilon,\omega_0}.
\end{align}
We may now define the so-called partial
input-output operator $\Lambda$ as 
\begin{align}\label{Def: input-output map}
\begin{aligned}
\Lambda:{\mathbf H^1(\Omega)\times \mathbf L^2(\Omega)\times \mathbf H^1(\Sigma)}&\mapsto\ { \mathbf H^1(\Omega)\times \mathbf L^2(\Sigma_{-,\varepsilon/2,\og_0})}\\
\Lambda(\overrightarrow{\psi_0},\overrightarrow{\psi_1},\overrightarrow{f})&:=\left(\us(T,\cdot),\,\dd_\nu \us|_{\Sigma_{-,\varepsilon/2,\og_0}}\right),
\end{aligned}
\end{align}
where $\overrightarrow{u}$ be the solution to IBVP \eqref{Govern IBVP} and Neumann boundary data is measured only on $\Sigma_{-,\varepsilon/2,\og_0}$. Since $\Lambda$ is a continuous operator, we denote its operator norm by $\|\cdot\|$. For arbitrary $M_0, \alpha>1$, let $\mc{T}(M_0,\alpha)$ denote the admissible class of functions, given by 
\begin{align*}
\mc{T}(M_0,\alpha):=\Big\{(\mc{A},\mc{B},q):\,A^i_0\in C_c^\infty(Q,\mathbb{R}),\,\overrightarrow{A}^i\in C_c^\infty(Q, \mathbb{R}^n),\, q_{ij}\in L^\infty(Q), \,\m{for}\, 1\leq i,j\leq m, \m{such that}\\
\|\mc{A}\|_{\mathbf H^{\frac{n+1}{2}+\alpha}(Q)}+\|\mc{B}\|_{\mathbf H^{\frac{n+1}{2}+\alpha}(Q)}+\|q\|_{\mathbf H^{\frac{n+1}{2}+\alpha}(Q)}\le M_0\Big\}.
\end{align*}
The main result of this paper is as follows.
\begin{theorem}[Statement of main result]\label{main theorem} 
For each $\ell=1,2$, let
\begin{align*}
\mathcal{A}^{(\ell)}=\diag\Big(A^{(\ell),1}_0,\,A^{(\ell),2}_0,\,\cdots,\,A^{(\ell),m}_0\Big)~\m{and}~
\mathcal{B}^{(\ell)}=\diag\Big(\overrightarrow{A}^{(\ell),1},\,\overrightarrow{A}^{(\ell),2},\,\cdots,\,\overrightarrow{A}^{(\ell),m}\Big).
\end{align*}
Let $(\mc{A}^{(\ell)},\mc{B}^{(\ell)},q^{(\ell)})\in \mc{T}(M_0,\alpha)$ 
such that
\begin{align}\label{Def: extra condition}
    \lVert (A^{(2),j}_0, \overrightarrow{A}^{{(2),j}})- (A^{(2),k}_0, \overrightarrow{A}^{{(2),k}})\rVert_{L^{\infty}(Q)}\leq \lVert(\mc{A}^{(1)},\mc{B}^{(1)})-(\mc{A}^{(2)},\mc{B}^{(2)})\rVert_{\mathbf{L}^{\infty}(Q)},
\end{align}
 for all $1\leq j,k\leq m$ and $\Lambda_\ell$ be the partial input-output operator corresponding to $\mc{L}_{\mc{A}^{(\ell)},\mc{B}^{(\ell)},q^{(\ell)}}$. 
Furthermore,
assume that \[\div_{(t,x)}(\mc{A}^{(1)}, \mc{B}^{(1)})=\div_{(t,x)}(\mc{A}^{(2)}, \mc{B}^{(2)}).\] Then the following estimates	
\begin{align} 
&\|(\mc{A}^{(2)},\mc{B}^{(2)})-(\mc{A}^{(1)},\mc{B}^{(1)})\|_{\LS^{\infty}(Q)}\le C\left(\|\Lambda_1-\Lambda_2\|^{\mu_1}+\big|\log|\log\|\Lambda_1-\Lambda_2\||\big|^{-\mu_2}\right),\label{main result 1}\\
&\|q^{(2)}-q^{(1)}\|_{\LS^{\infty}(Q)}\le C\left(\|\Lambda_1-\Lambda_2\|^{\mu_3}+\Big|\log\big|\log|\log\|\Lambda_1-\Lambda_2\||\big|\Big|^{-\mu_4}\right),\label{main result 2}
\end{align}	
 hold, where constants $C,\mu_1,\mu_2,\mu_3$ and $\mu_4>0$ are depending only on  $Q,m,\,\alpha\,\m{and}~M_0$.		
\end{theorem}
\begin{remark}
    The additional hypothesis \eqref{Def: extra condition} is required in determining the stable estimate for the off-diagonal components of the potential term $q$. However, for the diagonal components, this condition holds trivially. An analogue assumption for reaction-convection-diffusion systems is considered in \cite{bhardwaj2026stabledetermination}.
\end{remark}

\subsection{Existing papers}
The inverse problem of recovering coefficients in hyperbolic equations has received considerable attention in recent years. We begin by recalling some foundational results concerning inverse problems for hyperbolic equations with time-independent coefficients. Rakesh and Symes \cite{rakesh1988uniqueness} proved that the scalar potential in the wave equation can be uniquely determined from Neumann-to-Dirichlet (NtD) map by employing complex geometric optics (CGO) solutions. This CGO construction originates from the work of Sylvester and Uhlmann \cite{sylvester1987global} in the context of the conductivity equation.
Subsequently, Isakov \cite{isakov1991inverse} extended this uniqueness result to the case of perturbations involving time derivatives (damping term).

For the time-dependent coefficients, reconstruction of a potential within a restricted domain, given that the potential is known outside the domain, has been established in \cite{ramm1991property}. However, in an infinite cylinder, the uniqueness problem is settled by Ramm and Sj\"ostrand \cite{ramm1991inverse}, later generalized in \cite{Salazar_2013, Salazar_2014}. In \cite{stefanov1998stability, stefanov2005stable}, Stefanov and Uhlmann derived stable determination of the Riemannian metric from information encoded in the Dirichlet-to-Neumann (DtN) map. In \cite{kian2013stability}, Kian investigated the stable recovery of the zeroth-order coefficient in an infinite waveguide from the DtN map. In \cite{kian2016recovery}, Kian studied the wave equation with a first-order time-derivative perturbation and proved uniqueness for the corresponding inverse problem.
Further, in \cite{kian2017unique}, a unique identification result was proved for a time-dependent potential $q(t,x)$ from partial boundary measurements. In a Riemannian setup, Kian and Oksanen \cite{kian2019recovery} proved the global uniqueness of $q(t,x)$ from the Cauchy data set.
For more results on stability in the determination of the coefficients in a single wave equation from boundary measurements, we refer \cite{bellassoued2004uniqueness, bellassoued2006lipschitz, bellassoued2017optimal, bellassoued2008stability,
bellassoued2019inverse, bellassoued2019stability, kian2016stability, salazar2014stability, senapati2020stability,isakov1992stability} and the references therein. Recently, Kumar et al. \cite{kumar2026partial} established a unique coefficient identification result in a semilinear damped wave operator using the partial DtN map, whereas in \cite{bhardwaj2026reconstruction}, the authors explicitly reconstructed the coefficients for the same operator, but the DtN map is accessible on the entire boundary. For further studies on inverse problems for nonlinear hyperbolic PDEs with partial data, we refer to \cite{kian2021determination,feizmohammadi2021inverse,liu2025partial}.

In \cite{krishnan2020inverse}, uniqueness was obtained for an inverse problem associated with the relativistic Schr\"odinger equation. \cite{senapati2020stability} settled the stability of determining the coefficient of the relativistic Schr\"odinger equation. 
The present work extends \cite{senapati2020stability}, where the author considered a single relativistic Schr\"odinger equation; here, we establish stability estimates for a coupled system of relativistic Schr\"odinger equations. 
The inverse problem of reconstructing matrix-valued potentials has received comparatively less attention than the corresponding scalar case. Avdonin et al. \cite{avdonin1992boundary} established recovery of a time-independent potential for a wave equation. For the Schr\"odinger equation posed in a bounded domain under a convexity assumption, Eskin and Ralston \cite{eskin2003inverse} determined both the zeroth-order and first-order matrix-valued coefficients from the DtN map. In one spatial dimension, Khanfer et al. \cite{khanfer2019inverse} proved uniqueness of a time-dependent matrix potential under suitable assumptions. 
Mishra and Vashisth \cite{mishra2021determining} demonstrated uniqueness for a system of wave equations with a zeroth-order perturbation using partial boundary data, whereas \cite{kumar2024stable} derived a stability estimate for an analogous system in an infinite waveguide. For an inverse problem associated with a coupled system of wave equations using point source-receiver data, we refer to \cite{bhardwaj2026inverseproblemscoupledwave}.
More recently, \cite{BenFraj2025} and \cite{KumarSarkar2026} proved that a time-dependent matrix potential for the wave equation can be stably determined from arbitrary measurements and input-output map with partial data respectively.

In the present work, we investigate an inverse problem: stably determining coefficients from an input-output operator when the Neumann boundary measurements are available only on a subset of the boundary. More precisely, we demonstrate the Theorem \ref{main theorem}. The proof of Theorem~\ref{main theorem} relies on deriving a suitable integral identity and constructing a suitable choice of solution (known as geometric optics solution) to IBVP \eqref{Govern IBVP} and its $ L^2$-adjoint problem, and finally, using  the Carleman estimates in the aforementioned identity.
The coupling between the system components and partial boundary observations are considered in 
\cite{mishra2021determining, kumar2024stable} and \cite{KumarSarkar2026} respectively. However, the main challenge in this work lies in recovering simultaneously the two coefficient classes; first-order coefficients $(\mathcal{A},\mathcal{B})$ and a matrix potential $q$. It is observed that the stability estimates get deteriorated from log–log for the first-order coefficients to log–log–log for the zeroth order matrix potential. This reflects the additional difficulty of extracting the zeroth-order matrix potential from the available data.

The remainder of this paper is organized as follows. In Section \ref{Section: 2}, we discuss the existence and uniqueness of the solution to the IBVP \eqref{Govern IBVP} along with the energy estimate and bound on the normal derivative. Section \ref{Carleman Estimate} is devoted to deriving the boundary Carleman estimate for the operator $\mathcal{L}_{\mc{A},\mc{B},q}$. Next, in Section \ref{section-GO solution}, we construct the geometric optics solutions associated to the operator $\mathcal{L}_{\mc{A},\mc{B},q}$ and its adjoint operator. Finally, in Section \ref{Section: Proof of main result}, we first derive an integral identity, and then we use the aforementioned geometric optics solution and Carleman estimate to derive the stability estimates \eqref{main result 1} and \eqref{main result 2} in Subsection \ref{Stability estimate for matrix potentials} and \ref{Stability estimate for potential matrix q}, respectively.

\section{Well-posedness of the forward problem}\label{Section: 2}
In this section, we establish the existence, uniqueness, and stability estimate for the IBVP \eqref{Govern IBVP}. Motivated by the approach of  \cite[Theorem 2.30]{kachalov2001inverse} and \cite[Section 2]{lasiecka1986non}, we prove the following result.
\begin{theorem}\label{Theorem:well_posedness}
   Let $\overrightarrow{\psi_0}\in \HS^1(\Omega),\,\overrightarrow{\psi_1}\in \LS^2(\Omega)$ and $\overrightarrow{f}\in \HS^1(\Sigma)$ and satisfies the compatibility criteria, $\overrightarrow{\psi_0}|_{\partial\Omega}=\overrightarrow{f}(0,\cdot)|_{\partial\Omega}$. Then the IBVP \eqref{Govern IBVP} has a unique solution $$\overrightarrow{u}\in C^1\big([0,T];\LS^2(\Omega)\big)\cap C\big([0,T];\HS^1(\Omega)\big)$$ and there exists a constant $C>0$ depending on $\Omega,\,T,\,\mc{A},\,\mc{B}$ and $q$, such that for any $t\in[0,T]$, the following estimate holds
\begin{align}\label{regularity}
\|\partial_\nu \overrightarrow{u}\|_{\LS^2(\Sigma)}+\|\partial_t \overrightarrow{u}(t,\cdot)\|_{\LS^2(\Omega)}+\|\overrightarrow{u}(t,\cdot)\|_{\HS^1(\Omega)}\leq C\left(\|\overrightarrow{\psi_0}\|_{\HS^1(\Omega)}+\|\overrightarrow{\psi_1}\|_{\LS^2(\Omega)}+\|\overrightarrow{f}\|_{\HS^1(\Sigma)}\right).
\end{align}
\end{theorem}
\begin{proof}
To establish the existence and uniqueness of the solution to the initial-boundary value problem (IBVP) \eqref{Govern IBVP}, we first decompose
\[
\us(t,x) := \us^{(1)}(t,x) + \us^{(2)}(t,x).
\]
We then determine the functions $\us^{(1)}$ and $\us^{(2)}$ as solutions of the following IBVPs.
\begin{equation}\label{gov_eqn_for_u1}
\begin{cases}
  \square \us^{(1)} + \mathcal{A}(t,x) \,  \partial_t\us^{(1)} -\mathcal{B} (t,x)\cdot \nabla_x\us^{(1)}+\widetilde{\mathcal{A}}(t,x)\us^{(1)}=\overrightarrow{0},&(t,x)\in Q,\\
  \overrightarrow {u}^{(1)}(0,x)=\overrightarrow {\psi_0}(x),~~ \partial_t \overrightarrow {u}^{(1)}(0,x) = \overrightarrow {\psi_1}(x), &\quad x\in\Omega,\\
\overrightarrow {u}^{(1)}(t,x) = \overrightarrow {f}(t,x), & (t,x)\in  \Sigma,
\end{cases}
\end{equation}
and 
\begin{equation}\label{gov_eqn_for_u2}
\begin{cases}
  \mathcal{L}_{\mathcal{A},\mathcal{B},q}\us^{(2)}=-q(t,x)\us^{(1)},&(t,x)\in Q,\\
  \overrightarrow {u}^{(2)}(0,x)=\overrightarrow {0},~~ \partial_t \overrightarrow {u}^{(2)}(0,x) = \overrightarrow {0}, &\quad x\in\Omega,\\
\overrightarrow {u}^{(2)}(t,x) = \overrightarrow {0}, & (t,x)\in  \Sigma,
\end{cases}
\end{equation}
respectively. The existence and uniqueness of the solution $\us^{(1)}$ for the hyperbolic IBVP \eqref{gov_eqn_for_u1} can be guaranteed by \cite{kachalov2001inverse}. In addition, we have
\begin{align}\label{Existence: u_1}
\us^{(1)}\in C^1\big([0,T];\LS^2(\Omega)\big)\cap C\big([0,T];\HS^1(\Omega)\big),\quad \dd_\nu \us^{(1)}\in \LS^2(\Sigma)\end{align} and there exists a constant $C=C(\mc{A},\mc{B},\Omega,T)>0$ such that
\begin{align}\label{energy_estimate_u1}
\|\partial_\nu \us^{(1)}\|_{\LS^2(\Sigma)}+\|\us^{(1)}\|_{\HS^1(Q)}\leq C\left(\|\overrightarrow{\psi_0}\|_{\HS^1(\Omega)}+\|\overrightarrow{\psi_1}\|_{\LS^2(\Omega)}+\|\overrightarrow{f}\|_{\HS^1(\Sigma)}\right).\end{align}
Furthermore, for the existence and uniqueness of $\us^{(2)}$, we now introduce the sesquilinear form
$B(t;\cdot,\cdot)$ associated to IBVP
\eqref{gov_eqn_for_u2}.  Define
$B(t;\cdot,\cdot):\HS_0^1(\Omega)\times \HS_0^1(\Omega)\rightarrow \mathbb{R}$ by
\begin{align}\label{bilinear_map}
B[t;\gs,\hs]:=\int_{\Omega}\left[\nabla_x\gs\cdot\overline{\nabla_x\hs}-\big(\mc{B}(t,x)\cdot\nabla_x\gs\big)\cdot\overline{\hs}+\big(\widetilde{\mathcal{A}}(t,x)+q(t,x)\big)\gs\cdot\overline{\hs}\right]dx.
\end{align}
Using the Cauchy-Schwartz inequality, one can observe that the sesquilinear form $B[t;\cdot,\cdot]$ is continuous on $\HS_0^1(\Omega)\times \HS_0^1(\Omega)$. More precisely, we have 
\begin{align}\label{Bilinear_is_bdd}
|B[t;\gs,\hs]|\leq C\|\gs\|^2_{\HS_0^1(\Omega)}\|\hs\|^2_{\HS_0^1(\Omega)},
\end{align}
where constant $C=C(\mc{A},\,\mc{B},\,q,\,\Omega,\,T)>0$.
Also, we have
\begin{align*}
\Re B[t;\gs,\gs]\geq\|\nabla_x\gs\|^2_{\LS^2(\Omega)}-\Big(\|\widetilde{\mathcal{A}}\|_{\mathbf L^\infty(Q)}+\| q\|_{\mathbf L^\infty(Q)}\Big)\|\gs\|^2_{\LS^2(\Omega)},
\end{align*}
where $\Re$ denotes the real part of the underlying function.
Choosing $\lambda>\|\widetilde{\mathcal{A}}\|_{\mathbf L^\infty(Q)}+\|q\|_{\mathbf L^\infty(Q)}$ in the above equation, we get 
\begin{align}\label{bilinear_coercivity}
\Re B[t;\gs,\gs]+\lambda\|\gs\|^2_{\LS^2(\Omega)}\geq\alpha\|\gs\|^2_{\HS^1(\Omega)}\quad \text{for some constant}\,\alpha>0.
\end{align} 
Therefore, using the Theorem $8.1$ together with Remark $8.1$ of Chapter $3$  in \cite{lions2012non} (see also \cite{lions2013equations}), we observe that the IBVP \eqref{gov_eqn_for_u2} admits a unique solution $\us^{(2)}\in C^1\big([0,T];\LS^2(\Omega)\big)\cap C\big([0,T];\HS^1(\Omega)\big)$ and it satisfies the following estimate 
\begin{align}\label{H1_estimate_for_u2}
\lVert\us^{(2)}\rVert_{\HS^1(Q)}\leq C \lb \lVert\overrightarrow{\psi_0}\rVert_{\HS^1(\Omega)}+ \lVert\overrightarrow{\psi_1}\rVert_{\LS^2(\Omega)}+\lVert\overrightarrow {f}\rVert_{\HS^1(\Sigma)}\rb.
\end{align}
To complete the proof of Theorem \ref{Theorem:well_posedness}, it remains to verify that $\dd_\nu\us^{(2)}\in \LS^2(\Sigma)$. To this end, we follow the approach proposed in \cite{mishra2021determining}. Fix an arbitrary point $x\in\dd\Omega$ and denote by $\nu(x)$ the outward unit normal vector to $\dd\Omega$ at $x$. We then extend $\nu$ to a (sufficiently smooth) vector field on $\overline{\Omega}$, and denote this extension by $\vartheta(x)$. Next, we multiply the left-hand side of IBVP
\eqref{gov_eqn_for_u2} with $(T-t)\vartheta(x)\cdot\nabla_{x}\us^{(2)}$ and integrate it over $Q$, we have
\begin{align*}
\int\limits_{Q}\lb T-t\rb  \Lc_{\mc{A},\mc{B},q}&\us^{(2)}(t,x) \cdot \lb \vartheta(x)\cdot\nabla_{x}\us^{(2)}(t,x)\rb\D x \D t\\
=&\int\limits_{Q}\lb T
-t\rb\dd_{t}^{2}\us^{(2)}(t,x)\cdot \lb \vartheta(x)\cdot\nabla_{x}\us^{(2)}(t,x)\rb\D x \D t \\
&\quad -\int\limits_{Q}\lb T-t\rb\Delta_{x}\us^{(2)}(t,x)\cdot \lb \vartheta(x)\cdot\nabla_{x}\us^{(2)}(t,x)\rb\D x \D t\\
&\quad +\int\limits_{Q}\lb T-t\rb \mc{A}(t,x)\dd_t\us^{(2)}(t,x)\cdot \lb \vartheta(x)\cdot\nabla_{x}\us^{(2)}(t,x)\rb\D x \D t\\
&\quad -\int\limits_{Q}\lb T-t\rb \lb\mc{B}(t,x)\cdot\nabla_x\us^{(2)}(t,x)\rb\cdot \lb \vartheta(x)\cdot\nabla_{x}\us^{(2)}(t,x)\rb\D x \D t\\
&\quad +\int\limits_{Q}\lb T-t\rb \lb \widetilde{\mc{A}}+q(t,x)\rb\us^{(2)}(t,x)\cdot \lb \vartheta(x)\cdot\nabla_{x}\us^{(2)}(t,x)\rb\D x \D t\\
=& \sum_{j=1}^{m}\int\limits_{Q}\lb T-t\rb\dd_{t}^{2}u^{(2)}_j(t,x)\lb \vartheta(x)\cdot\nabla_{x}u^{(2)}_j(t,x)\rb\D x\D t\\
&\quad -\sum_{j=1}^{m}\int\limits_{Q}\lb T-t\rb\Delta_{x}u^{(2)}_j(t,x)\lb \vartheta(x)\cdot\nabla_{x}u^{(2)}_j(t,x)\rb\D x\D t\\
&\quad +\sum_{j=1}^{m}\int\limits_{Q}\lb T-t\rb A^j_0(t,x)\dd_t u^{(2)}_j(t,x)\lb \vartheta(x)\cdot\nabla_{x}u^{(2)}_j(t,x)\rb\D x\D t\\
&\quad -\sum_{j=1}^{m}\int\limits_{Q}\lb T-t\rb\lb \overrightarrow{A^{j}}(t,x)\cdot\nabla_{x}u^{(2)}_j(t,x)\rb\lb \vartheta(x)\cdot\nabla_{x}u^{(2)}_j(t,x)\rb\D x\D t\\
&\quad +\sum_{j=1}^{m}\int\limits_{Q}\lb T-t\rb  \widetilde{A^j}(t,x) u^{(2)}_j(t,x)\lb \vartheta(x)\cdot\nabla_{x}{u^{(2)}_j}(t,x)\rb\D x\D t\\
&\quad +\sum_{i,j=1}^{m}\int\limits_{Q}\lb T-t\rb q_{ij}(t,x)u^{(2)}_j(t,x)\lb \vartheta(x)\cdot\nabla_{x}{u^{(2)}_i}(t,x)\rb\D x\D t\\
:\triangleq& I+II+III+IV+V+VI.
\end{align*}
By multiplying the right-hand side of the IBVP \eqref{gov_eqn_for_u2} by $(T - t)\,\vartheta(x)\cdot\nabla_{x}\us^{(2)}$, integrating the resulting expression over $Q$, and subsequently comparing it with the preceding identity, we obtain
\begin{align}\label{integral_identity_1}
I+II+III+IV+V+VI=-\int\limits_{Q}\lb T-t\rb q(t,x)\us^{(1)}(t,x)\cdot \lb \vartheta(x)\cdot\nabla_{x}\us^{(2)}(t,x)\rb\D x \D t.
\end{align}
Using integration by parts, we observe that $I$ reduces to
\begin{align*}
\begin{aligned}
I&=-T\sum_{j=1}^{m}\int\limits_{\Omega} \dd_{t}u^{(2)}_j(0,x)\lb \vartheta(x)\cdot\nabla_{x}u^{(2)}_j(0,x)\rb\D x+\sum_{j=1}^{m}\int\limits_{Q}\dd_{t}u^{(2)}_j(t,x)\lb \vartheta(x)\cdot\nabla_{x}u^{(2)}_j(t,x)\rb \D x\D t\\
&\qquad \qquad-\sum_{j=1}^{m}\int\limits_{Q}\lb T-t\rb\dd_{t}u^{(2)}_j(t,x)\lb \vartheta(x)\cdot\nabla_{x}\dd_{t}u^{(2)}_j(t,x)\rb \D x\D t\\
&=-T\int\limits_{\Omega} \dd_{t}\us^{(2)}(0,x)\cdot\lb \vartheta(x)\cdot\nabla_{x}\us^{(2)}(0,x)\rb\D x+\int\limits_{Q}\dd_{t}\us^{(2)}(t,x)\cdot\lb \vartheta(x)\cdot\nabla_{x}\us^{(2)}(t,x)\rb\D x\D t\\
&\qquad \qquad  -\int\limits_{Q}\frac{(T-t)}{2} \nabla_{x}\cdot\lb \vartheta(x)\lvert\dd_{t}\us^{(2)}(t,x)\rvert^{2}\rb\D x\D t +\int\limits_{Q}\frac{(T-t)}{2}\lvert\dd_{t}\us^{(2)}(t,x)\rvert^{2}\nabla_{x}\cdot\vartheta(x)\D x \D t.
\end{aligned}
\end{align*}
Furthermore, by applying the Gauss divergence theorem along with the initial and boundary conditions prescribed in the IBVP \eqref{gov_eqn_for_u2}, the foregoing expression simplifies to
\begin{align}\label{integral_identity_A1}
\begin{aligned}
I&=\int\limits_{Q}\dd_{t}\us^{(2)}(t,x)\cdot\lb \vartheta(x)\cdot\nabla_{x}\us^{(2)}(t,x)\rb\D x\D t +\int\limits_{Q}\frac{T-t}{2}\lvert\dd_{t}\us^{(2)}(t,x)\rvert^{2}\nabla_{x}\cdot\vartheta(x)\D x \D t.
\end{aligned}
\end{align}
Similarly, using integration by parts, $II$ reduces to
\begin{align*}
II&=-\sum_{j=1}^{m}\int\limits_{Q}\lb T-t\rb\Delta_{x}u^{(2)}_j(t,x)\lb \vartheta(x)\cdot\nabla_{x}u^{(2)}_j(t,x)\rb\D x\D t\\
&=-\sum_{j,k=1}^{m}\int\limits_{Q}(T-t)\sum_{l=1}^{n}\dd_{k}^{2}u^{(2)}_j(t,x)\vartheta_{l}(x)\dd_{l}u^{(2)}_j(t,x)\D x\D t\\
&=-\sum_{j=1}^{m}\int\limits_{Q}(T-t)\nabla_{x}\cdot \lb \nabla_{x}u^{(2)}_j(t,x)\vartheta(x)\cdot\nabla_{x}u^{(2)}_j(t,x)\rb\D x\D t\\
&\qquad-\int\limits_{Q}\frac{T-t}{2}\nabla_{x}\cdot\vartheta(x)\lvert\nabla_{x}\us^{(2)}(t,x)\rvert^{2}\D x\D t\\
&\qquad+\sum_{j,k=1}^{m}\int\limits_{Q}(T-t)\sum_{l=1}^{n}\dd_{k}u^{(2)}_j(t,x)\dd_{k}\vartheta_{l}(x)\dd_{l}u^{(2)}_j(t,x)\D x\D t\\
&\qquad+\int\limits_{Q}\frac{T-t}{2}\nabla_{x}\cdot\lb \vartheta(x)\lvert\nabla_{x}\us^{(2)}(t,x)\rvert^{2}\rb\D x \D t.\\
\end{align*}
Next, we impose the Gauss divergence theorem in the above expression and use the vanishing boundary condition  to arrive at
\begin{equation}\label{integral_identity_A2}
\begin{aligned}
II&=-\int\limits_{\Sigma}\frac{T-t}{2}\lvert\dd_{\nu}\us^{(2)}(t,x)\rvert^{2}\D S_{x}\D t+\sum_{j,k=1}^{m}\int\limits_{Q}(T-t)\sum_{l=1}^{n}\dd_{k}u^{(2)}_j(t,x)\dd_{k}\vartheta_{l}(x)\dd_{l}u^{(2)}_j(t,x)\D x\D t\\
&\qquad \qquad -\int\limits_{Q}\frac{T-t}{2}\nabla_{x}\cdot\vartheta(x)\lvert\nabla_{x}\us^{(2)}(t,x)\rvert^{2}\D x\D t.
\end{aligned}
\end{equation}
Taking into account \eqref{integral_identity_A1}, \eqref{integral_identity_A2} in \eqref{integral_identity_1}, we have
\begin{align*}
    -&\int\limits_{\Sigma}\frac{T-t}{2}\lvert\dd_{\nu}\us^{(2)}(t,x)\rvert^{2}\D S_{x}\D t \\
    & \quad=-\sum_{j,k=1}^{m}\int\limits_{Q}(T-t)\sum_{l=1}^{n}\dd_{k}u^{(2)}_j(t,x)\dd_{k}\vartheta_{l}(x)\dd_{l}u^{(2)}_j(t,x)\D x\D t
    + \int\limits_{Q}\frac{T-t}{2}\nabla_{x}\cdot\vartheta(x)\lvert\nabla_{x}\us^{(2)}(t,x)\rvert^{2}\D x\D t\\
    &\qquad -\int\limits_{Q}\dd_{t}\us^{(2)}(t,x)\cdot\lb \vartheta(x)\cdot\nabla_{x}\us^{(2)}(t,x)\rb\D x\D t -\int\limits_{Q}\frac{T-t}{2}\lvert\dd_{t}\us^{(2)}(t,x)\rvert^{2}\nabla_{x}\cdot\vartheta(x)\D x \D t\\
    &\qquad -\int\limits_{Q}\lb T-t\rb q(t,x)\us^{(1)}(t,x)\cdot \lb \vartheta(x)\cdot\nabla_{x}\us^{(2)}(t,x)\rb\D x \D t -III-IV-V-VI.
\end{align*}
Afterwards, using the Cauchy-Schwartz inequality in the right-hand side of the above identity, we end up with
\begin{align}\label{L2_norm_normal_derivative}
\begin{aligned}
\Big|\int\limits_{\Sigma}&\frac{T-t}{2}\lvert\dd_{\nu}\us^{(2)}(t,x)\rvert^{2}\D S_{x}\D t\Big|\\
&\leq C\int\limits_{Q}\lb\lvert\us^{(1)}(t,x)\rvert^{2}+\lvert\us^{(2)}(t,x)\rvert^{2}+ \lvert\dd_{t}\us^{(2)}(t,x)\rvert^{2}+\lvert\nabla_{x}\us^{(2)}(t,x)\rvert^{2}\rb\D x\D t.
\end{aligned}
\end{align}
Combining \eqref{L2_norm_normal_derivative} with \eqref{energy_estimate_u1} and \eqref{H1_estimate_for_u2}, to get
\begin{align*}
 \|\partial_\nu \us^{(2)}\|_{\LS^2(\Sigma)} \leq C \lb \lVert\overrightarrow{\psi_0}\rVert_{\HS^1(\Omega)}+ \lVert\overrightarrow{\psi_1}\rVert_{\LS^2(\Omega)}+\lVert\overrightarrow {f}\rVert_{\HS^1(\Sigma)}\rb.
\end{align*}
As a consequence, we have 
\begin{align}\label{Existence: u_2}
\us^{(2)}\in \CS^{1}\lb [0,T];\LS^{2}(\Omega)\rb\cap \CS\lb [0,T];\HS^{1}(\Omega)\rb  \mbox{and}\  \dd_{\nu}\us^{(2)}\in \LS^{2}(\Sigma).
\end{align}
Combining \eqref{Existence: u_1} and \eqref{Existence: u_2}, we conclude that
\begin{align*}
&\us\in \CS^{1}\lb [0,T];\LS^{2}(\Omega)\rb\cap \CS\lb [0,T];\HS^{1}(\Omega)\rb  \mbox{and}\  \dd_{\nu}\us\in \LS^{2}(\Sigma).
\end{align*}
Also, we have
\begin{align*}
 \lVert\dd_{\nu}\us\rVert_{\LS^{2}(\Sigma)}+\lVert\vec{u}\rVert_{\HS^{1}(Q)}\leq C \lb \lVert\overrightarrow{\psi_0}\rVert_{\HS^1(\Omega)}+ \lVert\overrightarrow{\psi_1}\rVert_{\LS^2(\Omega)}+\lVert\overrightarrow {f}\rVert_{\HS^1(\Sigma)}\rb.
\end{align*}
Hence the result follows.
\end{proof}

\section{Carleman estimate}\label{Carleman Estimate}
In the present section, first we establish the Carleman estimate for the operator
$\mathcal{L}_{\mc{A},\mc{B},q}$ and then  shift its index to obtain weighted $H^{-1}-H^{-2}$ Carleman estimate. For the wave operator subjected exclusively to a zeroth-order perturbation, a corresponding Carleman estimate has been established in \cite{kian2016stability} for the weight function $\varphi_\lambda = \lambda \,(t + x \cdot \o)$. However, we modify the weight function in the presence of a first-order perturbation so that it absorbs the corresponding first-order contribution.
The proof idea follows and is inspired by \cite{kian2016recovery, krishnan2020inverse, senapati2020stability}.  
\subsection{Boundary Carleman estimate}
The Carleman estimate enables us to manage the boundary terms across sections of the boundary where measurements are lacking. We give a detailed proof of the Carleman estimate which provides the necessary ingredients for the stability estimate in Theorem \ref{main theorem}. More precisely, we have the following theorem.
\begin{theorem}\label{Theorem: Boundary Carleman estimate}
Let  $(\mc{A},\mc{B},q)\in \mc{T}(M_0,\alpha)$ and $\omega\in \Sb^{n-1}$. Furthermore, let $\overrightarrow{v}\in \CS^{2}(\overline{Q})$ be an arbitrary vector field satisfying
\begin{align*}
\overrightarrow{v}|_{\Sigma}=\overrightarrow{0},\ \overrightarrow{v}|_{t=0}=\partial_{t}\overrightarrow{v}|_{t=0}=\overrightarrow{0}.
\end{align*}
Then there exist constants $\lambda_{0}, C > 0$, both depending only on $m, M_{0}, \alpha \text{ and } Q$, such that for every $\lambda \geq \lambda_{0}$, the following holds:
\begin{align}\label{boundary carleman estimate}
\begin{aligned}
&\int_{Q}e^{-2\ld(t+x\cdot \og)}\left(\ld^2|\vs(t,x)|^2+|\nabla_{(t,x)}\vs(t,x)|^2\right) dxdt+\ld \int\limits_{\Sigma_{+,\omega}}e^{-2\ld(t+x\cdot \og)}|\omega\cdot \nu(x)||\partial_{\nu}\vs|^2dS_xdt\\
&\quad +\ld\int_{\Omega} e^{-2\ld(T+x\cdot \og)}|\dd_t\vs(T,x)|^2\ dx\le C\int_{\Omega} e^{-2\lambda(T+x\cdot\omega)}\left(\lambda^3|\vs(T,x)|^2+\ld|\nabla_x{\vs(T,x)}|^2\right)dx\\
&\quad+C \bigg(\int_{Q}e^{-2\lambda(t+x\cdot\omega)}|\mathcal{L}_{\mathcal{A},\mathcal{B},q}\vs(t,x)|^2dxdt+\lambda\int\limits_{\Sigma_{-,\omega}}{ e^{-2\lambda(t+x\cdot\omega)}|\omega\cdot \nu(x)||\partial_{\nu}\vs|^2}~dS_xdt \bigg).
\end{aligned}
\end{align}
\end{theorem}

\begin{proof}
For $(t,x)\in Q$, we denote $\varphi(t,x):=t+x\cdot\omega$ and $\varphi_\lambda(t,x):=\lambda\,(t+x\cdot\omega)$, where $\o\in \mathbb{S}^{n-1}$. To absorb the first-order perturbation, we next define a perturbed weight function depending on an auxiliary parameter $\sigma$, where $\sigma$ is taken to be independent of $\lambda$ and will be fixed at a later stage. Specifically, for parameters satisfying $\lambda>\sigma>1$, we set
\[
 \varphi_{\lambda,\sigma}(t,x):=\varphi_{\lambda}(t,x)-\frac{\sigma t^2}{2}.
\]
Now if we denote $\overrightarrow{u}=e^{-\phi_{\lambda,\sigma}}\overrightarrow{v}$, we then have $\overrightarrow{u}\in \textbf{C}^2(\overline{Q})$. Also, we have
\begin{align*}
\overrightarrow{u}|_{\Sigma}=\overrightarrow{0},\ \overrightarrow{u}|_{t=0}=\partial_{t}\overrightarrow{u}|_{t=0}=\overrightarrow{0}.
\end{align*}
Next, we define the conjugated operator $\mathcal{L}_{\varphi_{\lambda,\sigma}}$ of $\mathcal{L}_{\mathcal{A},\mathcal{B},q}$ by \begin{align}
\mathcal{L}_{\varphi_{\lambda,\sigma}}\overrightarrow{u}:&=e^{-\varphi_{\lambda,\sigma}}\mathcal{L}_{\mathcal{A},\mathcal{B},q}e^{\varphi_{\lambda,\sigma}}\overrightarrow{u}\\&=\square \overrightarrow{u}-\sigma \overrightarrow{u}+(\lambda-\sigma t)^2\overrightarrow{u}-\lambda^2 \overrightarrow{u}\\&\qquad+2(\lambda -\sigma t)\partial_t\overrightarrow{u}-2\lambda \o\cdot \nabla_x\overrightarrow{u}\\&\qquad+\mathcal{A}\partial_t\overrightarrow{u}+(\lambda-\sigma t)\mathcal{A}\overrightarrow{u}-\mathcal{B}\cdot \nabla_x \overrightarrow{u}-\lambda(\o\cdot \mathcal{B})\overrightarrow{u}+(\widetilde{\mathcal{A}}+q)\overrightarrow{u}\\:& \triangleq I_1 \overrightarrow{u}+I_2 \overrightarrow{u}+I_3\overrightarrow{u}.
\end{align}
Now, using the triangle inequality, we observe that
\begin{align}\label{Def:Norm conjugated operator}
\begin{aligned}
\lVert\mathcal{L}_{\phi_{\lambda,\sigma}}\overrightarrow{u}\rVert^2_{\textbf{L}^2(Q)}&=\int_{Q} \lvert(I_1+I_2+I_3)\overrightarrow{u}\rvert^2~dxdt \\&\geq \dfrac{1}{2}\lVert(I_1+I_2) \overrightarrow{u}\rVert^2_{\textbf{L}^2(Q)}- \lVert I_3 \overrightarrow{u}\rVert^2_{\textbf{L}^2(Q)}\\&\geq \int_{Q} \Re\left(\left(I_1 \overrightarrow{u}\right)\cdot \overline{\left(
I_2 \overrightarrow{u}\right)}\right)~dxdt- \lVert I_3 \overrightarrow{u}\rVert^2_{\textbf{L}^2(Q)}.
\end{aligned}
\end{align}
We then further simplify the terms in the above expression to obtain a suitable lower bound for $\lVert\mathcal{L}_{\phi_{\lambda,\sigma}}\overrightarrow{u}\rVert^2_{\textbf{L}^2(Q)}$. To this end, we first simplify the first term on the right-hand side of the inequality above. Taking into account  $\overrightarrow{u}|_{\Sigma}=\overrightarrow{0},\ \overrightarrow{u}|_{t=0}=\partial_{t}\overrightarrow{u}|_{t=0}=\overrightarrow{0}$ and the divergence theorem, we achieve
\begin{align}
    J_1:&=\int_{Q} \Re\left(\left(I_1 \overrightarrow{u}\right)\cdot \overline{\left(
I_2 \overrightarrow{u}\right)}\right)~dxdt\\&=
 \int_Q (\lambda-\sigma t)\partial_t\left( \lvert \partial_t \overrightarrow{u}\rvert^2+\lvert \nabla_x \overrightarrow{u}\rvert^2
 \right)~dxdt-2\lambda \Re \int_{\O}\partial_t\overrightarrow{u}(T,x)\cdot (\overline{\o\cdot\nabla_x\overrightarrow{u}(T,x)})~dx\\&\qquad +\lambda \int_{Q}\o\cdot \nabla_x(\lvert \partial_t \overrightarrow{u}\rvert^2)~dxdt+\lambda \int_{\Sigma} (\o\cdot \nu)\lvert \partial_\nu \overrightarrow{u}\rvert^2~dS_x dt
 \\&\qquad +\int_{Q} \left(-\lambda \sigma +\sigma^2 t-\sigma^3t^3+3\lambda \sigma^2t^2-2\lambda^2\sigma t
 \right)\partial_t(\lvert \overrightarrow{u}\rvert^2) ~dxdt.
\end{align}
A further simplification will give us 
\begin{align}
    J_1&=\int_{\O}(\lambda-\sigma T)(\lvert \partial_t \overrightarrow{u}(T,x) \rvert^2+\lvert \nabla_x \overrightarrow{u}(T,x)\rvert^2)~dx +\sigma \int_{Q} (\lvert \partial_t \overrightarrow{u}\rvert^2+\lvert \nabla_x \overrightarrow{u}\rvert^2)~dxdt
\\ &\qquad -2\lambda \Re \int_{\O}\partial_t\overrightarrow{u}(T,x)\cdot (\overline{\o\cdot\nabla_x\overrightarrow{u}(T,x)})~dx+\lambda \int_{\Sigma} (\o\cdot \nu)\lvert \partial_\nu \overrightarrow{u}\rvert^2~dS_x dt\\ &\qquad +\int_{\O} (-\sigma^3T^3+3\lambda \sigma^2 T^2 +\sigma^2 T-2\lambda^2\sigma T-\lambda\sigma)\lvert \overrightarrow{u}(T,x)\rvert^2~dx\\ &\qquad -\int_{Q} (-3\sigma^3 t^2 +6\lambda \sigma^2 t+\sigma^2-2\lambda^2 \sigma)\lvert \overrightarrow{u}\rvert^2~dxdt.
\end{align}
Next, we determine suitable lower bounds for the coefficients appearing in the preceding expression, which are then used to derive a lower bound for $J_1$. Observe that
\begin{enumerate}
    \item If $\lambda\geq (\sigma +6\lambda \sigma T)^{1/2}$, then
    \[ 
  - ( -3\sigma^3 t^2 +6\lambda \sigma^2 t+\sigma^2-2\lambda^2 \sigma)\geq \lambda^2 \sigma.
    \]
    \item If $\lambda \geq \left(\sigma^2T^2+\frac{\lambda}{T}\right)^{1/2}$, then 
    \[-\sigma^3T^3+3\lambda \sigma^2 T^2 +\sigma^2 T-2\lambda^2\sigma T-\lambda\sigma\geq -3\lambda^2 \sigma T.\]
    \item If $\lambda \geq 2 \sigma T$, then
    \[\lambda -\sigma T\geq \frac{\lambda}{2}.\]
\end{enumerate}
Thus, we have 
\begin{align}\label{Bound: J_1}
\begin{aligned}
    J_1&\geq \frac{\lambda}{4}\int_{\O} \lvert \partial_t \overrightarrow{u}(T,x) \rvert^2~dx -\frac{7\lambda}{2}\int_{\O}\lvert \nabla_x \overrightarrow{u}(T,x)\rvert^2~dx+\sigma\int_{Q} (\lvert \partial_t \overrightarrow{u}\rvert^2+\lvert \nabla_x \overrightarrow{u}\rvert^2)~dxdt\\&\qquad+\lambda \int_{\Sigma} (\o\cdot \nu)\lvert \partial_\nu \overrightarrow{u}\rvert^2~dS_x dt+\lambda^2 \sigma \int_{Q} \lvert \overrightarrow{u}\rvert^2~dxdt-3\lambda^2 \sigma T \int_{\O}\lvert \overrightarrow{u}(T,x)\rvert^2~dx
    \end{aligned}
\end{align}
for $\lambda> \max \left\{(\sigma +6\lambda \sigma T)^{1/2}, \left(\sigma^2T^2+\frac{\lambda}{T}\right)^{1/2}, 2 \sigma T\right\}$.
Now, we seek an upper bound for the second term on the right-hand side of \eqref{Def:Norm conjugated operator}. Following Sobolev embedding, we obtain
\begin{align}\label{Bound: I3}
\begin{aligned}
    \lVert I_3 \overrightarrow{u}\rVert^2_{\textbf{L}^2(Q)}:&= \lVert\mathcal{A}\partial_t\overrightarrow{u}+(\lambda-\sigma t)\mathcal{A}\overrightarrow{u}-\mathcal{B}\cdot \nabla_x \overrightarrow{u}-\lambda(\o\cdot \mathcal{B})\overrightarrow{u}+(\widetilde{\mathcal{A}}+q)\overrightarrow{u}\rVert^2_{\textbf{L}^2(Q)}\\& \leq C\bigg(  M_0^2\lVert \partial_t\overrightarrow{u}\rVert^2_{\textbf{L}^2(Q)}+32\lambda^2M_0^2\lVert \overrightarrow{u}\rVert^2_{\textbf{L}^2(Q)}+32 m^2M_0^2\lVert \overrightarrow{u}\rVert^2_{\textbf{L}^2(Q)}\\&\qquad\quad +8M_0^2\lVert \nabla_x \overrightarrow{u}\rVert^2_{\textbf{L}^2(Q)}+326 M_0^2 \lVert \overrightarrow{u}\rVert^2_{\textbf{L}^2{(Q)}}+384M_0^4\lVert  \overrightarrow{u}\rVert^2_{\textbf{L}^2(Q)}\bigg)\\& \leq C(m,M_0,Q,\alpha)(\lVert \partial_t\overrightarrow{u}\rVert^2_{\textbf{L}^2(Q)}+\lambda^2\lVert \overrightarrow{u}\rVert^2_{\textbf{L}^2(Q)}+\lVert \overrightarrow{u}\rVert^2_{\textbf{L}^2(Q)}+\lVert \nabla_x \overrightarrow{u}\rVert^2_{\textbf{L}^2(Q)})
    \end{aligned}
\end{align}
for $\lambda \geq 2\sigma T.$ We then combine \eqref{Bound: J_1} and \eqref{Bound: I3} to achieve
\begin{align}
    \begin{aligned}
\lVert\mathcal{L}_{\phi_{\lambda,\sigma}}\overrightarrow{u}\rVert^2_{\textbf{L}^2(Q)}&\geq  \frac{\lambda}{4}\int_{\O} \lvert \partial_t \overrightarrow{u}(T,x) \rvert^2~dx -\frac{7\lambda}{2}\int_{\O}\lvert \nabla_x \overrightarrow{u}(T,x)\rvert^2~dx+\sigma\int_{Q} (\lvert \partial_t \overrightarrow{u}\rvert^2+\lvert \nabla_x \overrightarrow{u}\rvert^2)~dxdt\\&\qquad+\lambda \int_{\Sigma} (\o\cdot \nu)\lvert \partial_\nu \overrightarrow{u}\rvert^2~dS_x dt+\lambda^2 \sigma \int_{Q} \lvert \overrightarrow{u}\rvert^2~dxdt-3\lambda^2 \sigma T \int_{\O}\lvert \overrightarrow{u}(T,x)\rvert^2~dx\\&\qquad-C(m,M_0,Q,\alpha)(\lVert \partial_t\overrightarrow{u}\rVert^2_{\textbf{L}^2(Q)}+\lambda^2\lVert \overrightarrow{u}\rVert^2_{\textbf{L}^2(Q)}+\lVert \overrightarrow{u}\rVert^2_{\textbf{L}^2(Q)}+\lVert \nabla_x \overrightarrow{u}\rVert^2_{\textbf{L}^2(Q)}).
    \end{aligned}
\end{align}
By choosing $\sigma = C(m,M_0,Q,\alpha) + 6$, we consequently obtain the following bound:
\begin{align}
    \begin{aligned}
\lVert\mathcal{L}_{\phi_{\lambda,\sigma}}\overrightarrow{u}\rVert^2_{\textbf{L}^2(Q)}&\geq  \frac{\lambda}{4}\int_{\O} \lvert \partial_t \overrightarrow{u}(T,x) \rvert^2~dx -7\lambda\int_{\O}\lvert \nabla_x \overrightarrow{u}(T,x)\rvert^2~dx+2\int_{Q} (\lvert \partial_t \overrightarrow{u}\rvert^2+\lvert \nabla_x \overrightarrow{u}\rvert^2)~dxdt\\&\qquad+\lambda \int_{\Sigma} (\o\cdot \nu)\lvert \partial_\nu \overrightarrow{u}\rvert^2~dS_x dt+5\lambda^2  \int_{Q} \lvert \overrightarrow{u}\rvert^2~dxdt-6\lambda^3 \sigma T \int_{\O}\lvert \overrightarrow{u}(T,x)\rvert^2~dx.
    \end{aligned}
\end{align}
Subsequently, we substitute $\overrightarrow{u}$ by $e^{-\varphi_{\lambda,\sigma}}\overrightarrow{v}$ in the preceding expression and simplify to obtain
\begin{align}
    \begin{aligned}
       e^{\sigma T^2} \lVert e^{-\varphi_{\lambda}}\mathcal{L}_{\mathcal{A},\mathcal{B},q}\overrightarrow{v}\rVert^2_{\textbf{L}^2(Q)}&\geq \lambda \int_{\Sigma_+} e^{-2\varphi_\lambda}(\o\cdot \nu)\lvert \partial_\nu \overrightarrow{v}\rvert^2~dS_x dt-\lambda e^{\sigma T^2}\int_{\Sigma_-} e^{-2\varphi_\lambda}(\o\cdot \nu)\lvert \partial_\nu \overrightarrow{v}\rvert^2~dS_x dt\\&\quad+\frac{\lambda}{8}\int_{\O}e^{-2\varphi_{\lambda}}\lvert \partial_t \overrightarrow{v}(T,x)\rvert^2~dx-(6\sigma T+14 e^{\sigma T^2} +1)\lambda^3 \int_\O e^{-2\varphi_{\lambda}}\lvert \overrightarrow{v}(T,x)\rvert^2~dx\\&\quad+\int_{Q}e^{-2\varphi_{\lambda}} (\lvert \partial_t \overrightarrow{v}\rvert^2+\lvert \nabla_x \overrightarrow{v}\rvert^2)~dxdt+\lambda^2\int_{Q}e^{-2\varphi_{\lambda}} \lvert \overrightarrow{v}\rvert^2~dxdt\\&\quad-14\lambda e^{\sigma T^2}\int_\O e^{-2\varphi_{\lambda}}\lvert \nabla_x\cdot\overrightarrow{v}(T,x)\rvert^2~dx.
    \end{aligned}
\end{align}
Upon further simplification, and invoking the uniform bound $e^{\sigma T^{2}}$, the above expression reduces to
\begin{align}
\begin{aligned}
&\int_{Q}e^{-2\ld(t+x\cdot \og)}\left(\ld^2|\vs(t,x)|^2+|\nabla_{(t,x)}\vs(t,x)|^2\right) dxdt+\ld \int\limits_{\Sigma_{+,\omega}}e^{-2\ld(t+x\cdot \og)}|\omega\cdot \nu(x)||\partial_{\nu}\vs|^2dSdt\\
&\quad +\ld\int_{\Omega} e^{-2\ld(T+x\cdot \og)}|\dd_t\vs(T,x)|^2\ dx\le C\int_{\Omega} e^{-2\lambda(T+x\cdot\omega)}\left(\lambda^3|\vs(T,x)|^2+\ld|\nabla_x{\vs(T,x)}|^2\right)dx\\
&\quad+C \bigg(\int_{Q}e^{-2\lambda(t+x\cdot\omega)}|\mathcal{L}_{\mathcal{A},\mathcal{B},q}\vs(t,x)|^2dxdt+\lambda\int\limits_{\Sigma_{-,\omega}}{ e^{-2\lambda(t+x\cdot\omega)}|\omega\cdot \nu(x)||\partial_{\nu}\vs|^2}dS \bigg).
\end{aligned}
\end{align}
Thus the desired estimate \eqref{boundary carleman estimate} is obtained. 
\end{proof}
Afterwards, we state the interior estimate which follows from the Theorem \ref{Theorem: Boundary Carleman estimate}.
\begin{corollary} (\tb{Interior Carleman estimates})
Given $(\mc{A},\mc{B},q)\in \mc{T}(M_0,\alpha)$ and $\vs\in \CS_c^{\iy}(Q)$. Then there exist $C>0$ and $\ld_0>0$ depending only on $m,M_0,\alpha$ and $Q$ such that the following estimate
\begin{align}\label{boundary carleman estimate--1}
\int_{Q}  e^{-2\ld(t+x\cdot \og)}\left(\ld^2|\vs(t,x)|^2+|\nabla_{(t,x)}\vs(t,x)|^2\right) dxdt\le C \int_{Q} e^{-2\lambda(t+x\cdot\omega)}|\mathcal{L}_{\mathcal{A},\mathcal{B},q}\vs(t,x)|^2dxdt\nn,
\end{align}
 holds for any $\ld\ge\ld_0$.
\end{corollary}

\section{Geometric optics (GO) solutions}\label{section-GO solution}
In this section, we construct GO solutions to the operator 
$
\mathcal{L}_{\mathcal{A},\mathcal{B},q}\overrightarrow{u} = \overrightarrow{0}
$
as well as to its adjoint operator
$
\mathcal{L}^{*}_{\mathcal{A},\mathcal{B},q}\overrightarrow{v} = \overrightarrow{0}.
$
The derivation of these solutions relies on an interior Carleman estimate in negative-order Sobolev spaces, formulated in Lemmas 5.3 and 5.4 of \cite{kian2016recovery}. Before invoking these results, we introduce the function spaces employed throughout this section.

For any $k \in \mathbb{R}$, we define
\[
H^k_\lambda(\mathbb{R}^{1+n}) := \Big\{ u \in \mathcal{S}'(\mathbb{R}^{1+n}) : \big(|(\tau,\xi)|^2 + \lambda^2\big)^{\frac{k}{2}} \hat{u} \in L^2(\mathbb{R}^{1+n}) \Big\},
\]
with the norm
\begin{equation}
\|u\|^2_{H_{\lambda}^k(\mathbb{R}^{1+n})} := \int_{\mathbb{R}^{1+n}} \big(|(\tau,\xi)|^2 + \lambda^2\big)^k \lvert \hat{u}(\tau,\xi) \rvert^2 \, d\tau \, d\xi,
\label{parameter_norm}
\end{equation}
where $\mathcal{S}'$ denotes the space of tempered distributions on $\mathbb{R}^{1+n}$, and $\hat{u}$ is the space-time Fourier transform of $u$. Here we use the notation
\[
(\tau,\xi):= (\tau^2 + \lvert \xi \rvert^2)^{1/2}.
\]
In particular, we have 
\[\lVert u\rVert_{H^1_\lambda(Q)}^2:= \lambda^2\lVert u\rVert_{L^2{(Q)}}^2+\lVert \nabla_{(x,t)} u\rVert_{L^2{(Q)}}^2.\]

We are now in a position to state the existence of exponentially growing solutions to the equation \(\mathcal{L}_{\mathcal{A},\mathcal{B},q}\overrightarrow{u} = \overrightarrow{0}\), which is formulated in the following theorem.

\begin{proposition}\label{proposition for GO sol 1}
Let $(\mc{A},\mc{B},q)\in \mc{T}(M_0,\alpha)$, $\omega\in\mathbb{S}^{n-1}$, and $\phi_j\in C_c^\infty(\mathbb{R}^n)$ for each $1\leq j\leq m$. Define
\[
\mathcal{C}^{j}(t,x):=(A^j_0,\overrightarrow{A}^j)=(A^j_0,A^j_{1},A^j_{2},\cdots,A^j_{n}), \quad j=1,2,\dots,m.
\]
Then, there exist constants $C,\lambda_0>0$, depending only on $m,M_0,\alpha$ and $Q$  such that the operator
$
\mathcal{L}_{\mathcal{A},\mathcal{B},q}\overrightarrow{u}=\overrightarrow{0}
$
admits an exponentially growing GO solution of the form
\begin{equation}\label{GO solution for without adjoint operator}        
\overrightarrow{u_g}(t,x)=e^{\lambda(t+x\cdot\omega)}\Big(\overrightarrow{B_g}(t,x)+\overrightarrow{R}_{g,\lambda}(t,x)\Big),
\end{equation}
where the $j$-th component of $\overrightarrow{B_g}$ is given explicitly by
\[
B^j_g(t,x)=K^j\,\phi_j(x+t\omega)\exp\!\Bigg(-\int_0^t (1,-\omega)\cdot\mathcal{C}^{j}\big(s,x+(t-s)\omega\big)\,ds\Bigg),
\]
and $K^j$ is a constant.
Also, the remainder term $\overrightarrow{R}_{g,\lambda}$ belongs to $\HS^2(Q)$ satisfying
\begin{equation}
   \big\lVert \overrightarrow{R}_{g,\lambda}\big\rVert_{\HS^k(Q)}\leq C\,\lambda^{-1+k}\big\lVert\overrightarrow{\phi}\big\rVert_{\HS^3(\mathbb{R}^n)},\quad k=0,1,2,
\end{equation}
holds, where $\overrightarrow{\phi}=(\phi_1,\,\phi_2,\,\dots,\,\phi_m)^T$.
\end{proposition}
To this end, we require a Carleman estimate in negative-order Sobolev spaces. A relevant result is provided by (Lemma 5.4, \cite{kian2016recovery}), where such an estimate is established. However, an additional shift of order $-1$ is required in present setting, which can be derived by similar arguments, for the operator defined below.
\begin{lemma}\cite[Lemma 5.4]{kian2016recovery}\label{Lemma for GO sol}
For $(\mc{A},\mc{B},q)\in \mc{T}(M_0,\alpha)$, let us consider the conjugated operator $$P_{\mc{A},\mc{B},\ld,\og}=e^{-\ld(x\cdot\og+t)}(\mathcal{L}_{\mathcal{A},\mc{B},q}-q)e^{\ld(x\cdot\og+t)}.$$ There exist constants $C,\ld_0>0$ depending on $M_0,\alpha$ and $Q$ such that for $\us\in \CS_c^{\iy}(Q)$ and $\ld\ge\ld_0$
\begin{align*}
\|\us\|_{\HS^{-1}_{\ld}(\mb{R}^{1+n})}\le C\|P_{\mc{A},\mc{B},\ld,\og}\us\|_{\HS^{-2}_{\ld}(\mb{R}^{1+n})}.
\end{align*} 	
\end{lemma}
\begin{proof}[Proof of Proposition \ref{proposition for GO sol 1}] The proof is based on the Hahn-Banach and  Reis\'z representation theorem. 
 For this purpose, we consider the subspace 
 \[
 \mathcal{P}:=\{P_{-\mc{A},-\mc{B},-\ld,\og}\vs:\ \vs\in \CS_c^{\iy}(Q)\}
 \]
and 
define a linear map $\mc{S}: \mathcal{P}\to\R$ by
\begin{align}\label{Def: operator S}
\mc{S}(P_{-\mc{A},-\mc{B},-\ld,\og}\vs)=\int_{Q}\vs(t,x)\cdot\overrightarrow{f}(t,x)\ dxdt,\;~\overrightarrow{f}\in \HS^1_\ld(Q).
\end{align}
Following  Lemma \ref{Lemma for GO sol}, we conclude that  
\begin{align*}
|\mc{S}(P_{-\mc{A},-\mc{B},-\ld,\og}\vs)|\le \lVert \overrightarrow{v}\rVert_{\HS^{-1}_{\ld}(\mb{R}^{1+n})}\lVert \overrightarrow{f}\rVert_{\HS^{1}_{\ld}(Q)} \leq C
\|P_{-\mc{A},-\mc{B},-\ld,\og}\vs\|_{\HS^{-2}_{\ld}(\mb{R}^{1+n})}\|\overrightarrow{f}\|_{\HS^{1}_{\ld}(Q)}.
\end{align*}
Thus, $\mc{S}$ is a continuous linear operator.
Using the Hahn-Banach theorem, we now extend $\mc{S}$ to a continuous functional on $\HS^{-2}_{\ld}(\mb{R}^{1+n})$ and  still denoted by $\mc{S}$. Also, we have  $\|\mc{S}\|\le C\|\overrightarrow{f}\|_{\HS^{1}_{\ld}(Q)} $. Following the Reis\'z representation theorem, we have a unique $\us\in{\HS^{2}_{\ld}(\R^{1+n})}$ such that 
\begin{align}\label{exist1}
\mc{S}(\vs)=(\vs,\us)_{{\HS^{-2}_{\ld}(\R^{1+n})},{\HS^{2}_{\ld}(\R^{1+n})}},~~ \vs\in \CS_c^{\iy}(Q)
\end{align}
and 
\begin{align}
    \|\us\|_{{\HS^{2}_{\ld}(\R^{1+n})}}\le \|\mc{S}\|\le C\|\overrightarrow{f}\|_{\HS^{1}_{\ld}(Q)}.
\end{align}
Combining \eqref{Def: operator S} and \eqref{exist1} together, we have  
\begin{align}\label{exist2}
\begin{aligned}
\mc{S}(P_{-\mc{A},-\mc{B},-\ld,\og}\vs)=\int_{Q}\vs(t,x)\cdot\overrightarrow{f}(t,x)\ dxdt=(P_{-\mc{A},-\mc{B},-\ld,\og}\vs,\us)_{{\HS^{-2}_{\ld}(\R^{1+n})},{\HS^{2}_{\ld}(\R^{1+n})}},
\end{aligned}
\end{align}
for all $\vs\in \CS_c^{\iy}(Q)$.
From \eqref{exist2}, we conclude that there exist $\us\in{\HS^{2}_{\ld}(\R^{1+n})}$ such that 
\begin{align*}
P_{\mc{A},\mc{B},\ld,\og}\us=\overrightarrow{f}\m{ with  }\|\us\|_{{\HS^{2}_{\ld}(\R^{1+n})}}C\le\|\overrightarrow{f}\|_{\HS^{1}_{\ld}(Q)}.
\end{align*} 
Now, we seek for $\us_g$ solves $\mathcal{L}_{\mathcal{A}, \mathcal{B},q}\overrightarrow{u}=\overrightarrow{0}$, having the form
\begin{align}
&\quad \quad \us_g(t,x)=e^{\ld(x\cdot\og+t)}(\overrightarrow{B_g}(t,x)+\overrightarrow{R}_{g,\lambda}(t,x)).\nn
\end{align}
To do so, it is sufficient to find $\overrightarrow{B_g}$ and $\overrightarrow{R}_{g,\lambda}$ satisfying \[ e^{-\ld(x\cdot\og+t)}\mathcal{L}_{\mathcal{A},\mc{B},q}e^{\ld(x\cdot\og+t)}(\overrightarrow{B_g}(t,x)+\overrightarrow{R}_{g,\lambda}(t,x))=\overrightarrow{0}.\]
Equivalently,  \[P_{\mathcal{A},\mc{B},\ld,\og}\overrightarrow{R}_{g,\lambda}=-q\overrightarrow{R}_{g,\lambda}-P_{\mathcal{A},\mc{B},\ld,\og}\overrightarrow{B_g}-q\overrightarrow{B_g}=-q\overrightarrow{R}_{g,\lambda}-\Big(2\ld\big\{1,-\og\big\}\cdot\big\{\nabla_{t,x}B_g^j+\mc{C}^jB_g^j\big\}\Big)_{j=1}^m-\mc{L}_{\mc{A},\mc{B},q}\overrightarrow{B_g}.\]
Next, we set the term involving $\lambda$ to be 0, results in the following transport equation
\begin{align}
\qquad\qquad\Big(1,-\omega\Big)\cdot\Big(\nabla_{t,x}B_g^j(t,x)+\mc{C}^j(t,x)B_g^j(t,x)\Big)=0,\label{equation for GO solution}
\end{align}
for each $j=1,2,\dots,m$.
A particular solution to the above transport equation is given by 
\begin{align*}
B_g^j(t,x)=K^j\phi_j(x+t\omega)e^{-\int_0^t (1,-\omega)\cdot\mathcal{C}^{j}(s,x+(t-s)\omega)\,ds},
\end{align*}
where $K^j$ is $j^{th}$ component of an arbitrary constant vector $\overrightarrow{K}\in\mathbb{R}^m$ and $\phi_j\in C_c^\infty(\mathbb{R}^n)$ for each $1\leq j\leq m$. 
Thus, it remains to solve 
\begin{align}\label{exist3}
P_{\mathcal{A},\mc{B},\ld,\og}\overrightarrow{R}_{g,\lambda}=-q\overrightarrow{R}_{g,\lambda}-\mc{L}_{\mc{A},\mc{B},q}\overrightarrow{B_g},
\end{align}
where $\mc{L}_{\mc{A},\mc{B},q}\overrightarrow{B_g}\in \HS^1_\ld(Q)$ with $\|\mc{L}_{\mc{A},\mc{B},q}\overrightarrow{B_g}\|_{\HS^1_\ld(Q)}\le C\ld\|\overrightarrow{\phi}\|_{\HS^3(\R^n)}$.
Define $T(\overrightarrow{f})=\us$, where $P_{\mc{A},\mc{B},\ld,\og}\us=\overrightarrow{f}$. Then $\|T\|_{\HS^1_\ld(Q)\to \HS^1_\ld(Q)}\le \frac{C}{\ld}$ and $\|T\|_{\HS^1_\ld(Q)\to \HS^2_\ld(Q)}\le C$.
Thus the problem \eqref{exist3} reduces to finding $\overrightarrow{\tilde{f}}\in \HS^1_\ld(Q)$ such that 
\begin{align*}
    (I+qT)\overrightarrow{\tilde{f}}=-\mc{L}_{\mc{A},\mc{B},q}\overrightarrow{B_g}.
\end{align*}
The invertibility of $(I+qT)$ in $\HS^1_\ld(Q)$ is guaranteed for $\ld$  large enough. So, we can find $\overrightarrow{\tilde{f}}\in \HS^1_\ld(Q)$ and hence $\overrightarrow{R}_{g,\lambda}\in \HS^2_\ld(Q)$ which satisfies 
\begin{align}\label{remainder}
\|\overrightarrow{R}_{g,\lambda}\|_{\HS^2_\ld(Q)}\le C\|\overrightarrow{\tilde{f}}\|_{\HS^1_\ld(Q)}\le C\ld\|\overrightarrow{\phi}\|_{\HS^3(Q)}.
\end{align}
This completes the proof.
\end{proof}
Further, by performing a similar analysis to that of Proposition \ref{proposition for GO sol 1}, one can construct GO solution to  $\mathcal{L}^*_{\mathcal{A},\mathcal{B},q}$, the adjoint operator of $\mathcal{L}_{\mathcal{A},\mathcal{B},q}$, where 
\begin{align}\label{definition of adjoint}
\mathcal{L}^*_{\mathcal{A},\mathcal{B},q}\overrightarrow {u}(t,x) :
={\displaystyle  \begin{bmatrix} (\partial_t-A^1_0(t,x))^2u_1(t,x)-\sum\limits_{j=1}^n(\partial_j-A^1_j(t,x))^2u_1(t,x)+\sum\limits_{j=1}^{m}q_{j1}(t,x)u_j(t,x) \\  (\partial_t-A^2_0(t,x))^2u_2(t,x)-\sum\limits_{j=1}^n(\partial_j-A^2_j(t,x))^2u_2(t,x)+\sum\limits_{j=1}^{m} q_{j2}(t,x)u_j(t,x) \\ \vdots \\ (\partial_t-A^m_0(t,x))^2u_m(t,x)-\sum\limits_{j=1}^n(\partial_j-A^m_j(t,x))^2u_m(t,x)+\sum\limits_{j=1}^{m} q_{jm}(t,x)u_j(t,x) \end{bmatrix}}.
\end{align}
\begin{proposition}\label{proposition for GO sol 2}
Let $(\mc{A},\mc{B},q)\in \mc{T}(M_0,\alpha)$, $\omega\in\mathbb{S}^{n-1}$, and $\phi_j\in C_c^\infty(\mathbb{R}^n)$ for each $1\leq j\leq m$. Define
\[
\mathcal{C}^{j}(t,x):=(A^j_0,\overrightarrow{A}^j)=(A^j_0,A^j_{1},A^j_{2},\cdots,A^j_{n}), \quad j=1,2,\dots,m.
\]
Then, there exist constants $C, \lambda_0>0$, depending only on $m,M_0,\alpha$ and $Q$  such that the operator
$
\mathcal{L}^*_{\mathcal{A},\mathcal{B},q}\overrightarrow{v_d}=\overrightarrow{0}
$
admits an exponentially decaying GO solution of the form
  \begin{equation}\label{GO solution for adjoint operator}       \overrightarrow{v_d}(t,x)=e^{-\lambda(t+x\cdot\omega)}\Big(\overrightarrow{B_d}(t,x)+\overrightarrow{R}_{d,\lambda}(t,x)\Big),
    \end{equation}
where the $j$-th component of $\overrightarrow{B_d}$ is given explicitly by
\begin{align*}
    B^j_d(t,x)=K^j\phi_j(x+t\omega)e^{\int_0^t (1,-\omega)\cdot\mathcal{C}^{j}(s,x+(t-s)\omega)\,ds},
\end{align*}
and $K^j$ is a constant.
Also, the remainder term $\overrightarrow{R}_{d,\lambda}$ belongs to $\HS^2(Q)$ satisfying
\begin{equation}
   \big\lVert \overrightarrow{R}_{d,\lambda}\big\rVert_{\HS^k(Q)}\leq C\,\lambda^{-1+k}\big\lVert\overrightarrow{\phi}\big\rVert_{\HS^3(\mathbb{R}^n)},\quad k=0,1,2,
\end{equation}
holds with $\overrightarrow{\phi}=(\phi_1,\,\phi_2,\,\dots,\,\phi_m)^T$.
\end{proposition}

\section{Proof of Theorem \ref{main theorem}}\label{Section: Proof of main result}
This section is devoted to the derivation of the main result of the article. We begin by establishing the integral identity \eqref{integral identity for stability}, which, in combination with the GO solutions and the Carleman estimate, is subsequently employed to obtain the desired stability estimate. For $i=1,2$, let
\begin{equation*}
\begin{split}
&\mathcal{A}^{(i)}=\diag \Big(A^{(i),1}_0,\,A^{(i),2}_0,\,\cdots,\,A^{(i),m}_0\Big),\\&\mathcal{B}^{(i)}=\diag \Big(\overrightarrow{A}^{(i),1},\,\overrightarrow{A}^{(i),2},\,\cdots,\,\overrightarrow{A}^{(i),m}\Big),\\&
\mathcal{C}^{(i)}=\diag \Big(\left(A^{(i),1}_0,\overrightarrow{A}^{(i),1}\right),\,\left(A^{(i),2}_0,\overrightarrow{A}^{(i),2}\right),\,\cdots,\,\left(A^{(i),m}_0,\overrightarrow{A}^{(i),m}\right)\Big),
\\
&\widetilde{\mathcal{A}^{(i)}}=\diag \Big(\widetilde{A^{(i),1}},\,\widetilde{A^{(i),2}},\,\cdots,\,\widetilde{A^{(i),m}}\Big),
\end{split}
\end{equation*}
where \[\widetilde{A^{(i),r}}=\partial_tA^{(i),r}_0+(A^{(i),r}_0)^2-\nabla_x\cdot\overrightarrow{A}^{(i),r}-|\overrightarrow{A}^{(i),r}|^2, \qquad r=1,2,\cdots,m.
\] 
Let $\overrightarrow{u}^{(i)}$ be the solution to the following IBVP
\begin{equation}\label{IBVP for u_i}
\left\{\begin{array}{l c}  \mathcal{L}_{\mathcal{A}^{(i)},\mathcal{B}^{(i)},q^{(i)}}\overrightarrow{u}^{(i)}(t,x) = \overrightarrow {0}, &\quad (t,x)\in Q,\\
\overrightarrow{u}^{(i)}(0,x)=\overrightarrow {\psi}_0(x),~~ \partial_t \overrightarrow{u}^{(i)}(0,x) = \overrightarrow {\psi}_1(x), &\quad x\in\Omega,\\
\overrightarrow{u}^{(i)}(t,x) = \overrightarrow {f}(t,x), &\quad (t,x)\in  \Sigma.
\end{array}\right.
\end{equation}
 We set
\begin{align}\label{notation used in integral identity}
\begin{aligned}
\overrightarrow{u}(t,x)&:=(\overrightarrow{u}^{(1)}-\overrightarrow{u}^{(2)})(t,x),\quad q(t,x):=\big(q^{(2)}-q^{(1)}\big)(t,x),\quad \mathcal{A}(t,x):=\Big(\mathcal{A}^{(2)}-\mathcal{A}^{(1)}\Big)(t,x),\\
\widetilde{\mathcal{A}}(t,x)&:=\Big(\widetilde{\mathcal{A}^{(2)}}(t,x)-\widetilde{\mathcal{A}^{(1)}}(t,x)\Big),\quad \mathcal{B}(t,x):=\Big(\mathcal{B}^{(2)}(t,x)-\mathcal{B}^{(1)}(t,x)\Big).
\end{aligned}
\end{align}
Also, for any $g\in L^1(\mathbb{R}^{1+n})$, we define the space-time Fourier transform of $g$ by
    \begin{equation}\label{defn:Fourier transform}
        \widehat{g}(\tau,\xi):=\int\limits_{\mathbb{R}^{n}}\int\limits_{\mathbb{R}
        }g(s,y)e^{-i(s\tau+y\cdot\xi)}\D s\D y,\quad (\tau,\xi)\in\mathbb{R}^{1+n} .
    \end{equation}
We can now formulate the integral identity, as presented in the following proposition.

 \begin{proposition}\label{propintid}
For each $i = 1,2$, let $\big(\mathcal{A}^{(i)},\mathcal{B}^{(i)},q^{(i)}\big)\in \mathcal{T}(M_0,\alpha)$ denote two admissible triples of matrix-valued potentials.  
Consider $\overrightarrow{u}^{(i)}$ to be the solution to \eqref{IBVP for u_i} associated with the potential $\big(\mathcal{A}^{(i)},\mathcal{B}^{(i)},q^{(i)}\big)$, and let $\overrightarrow{v}$ be a solution to the adjoint problem
\begin{align}\label{Def: adjoint probelm}
\mathcal{L}^{*}_{\mathcal{A}^{(1)},\mathcal{B}^{(1)},q^{(1)}} \overrightarrow{v}(t,x) = \overrightarrow{0}
\quad \text{in } Q.
\end{align}
Then the following identity holds:
\begin{align}\label{integral identity for stability}
\begin{aligned}
\int\limits_Q \Big(\mathcal{A}(t,x)\partial_t\overrightarrow{u}^{(2)}(t,x)-\mathcal{B}(t,x)\cdot\nabla_x\overrightarrow{u}^{(2)}(t,x)+\widetilde{\mathcal{A}}(t,x)\overrightarrow{u}^{(2)}(t,x)+q(t,x)\overrightarrow{u}^{(2)}(t,x)\Big)\cdot\overline{\overrightarrow{v}(t,x)}\, dxdt\\
=\int\limits_\Omega \partial_t\overrightarrow{u}(T,x)\cdot\overline{\overrightarrow{v}(T,x)}\, dx-\int\limits_\Omega\overrightarrow{u}(T,x)\cdot\overline{ \partial_t\overrightarrow{v}(T,x)}\, dx
-\int\limits_{\Sigma}\partial_\nu\overrightarrow{u}(t,x)\cdot\overline{\overrightarrow{v}(t,x)}\,dS_xdt.
\end{aligned}
\end{align}
\end{proposition}

\begin{proof}
 Using the IBVP \eqref{IBVP for u_i} and notations introduced in \eqref{notation used in integral identity}, we observe that $\overrightarrow{u}$ satisfies the IBVP
\begin{align}\label{IBVP for u (defined as differences of solution)}\left\{
\begin{array}{lc}  \mathcal{L}_{\mathcal{A}^{(1)},\mathcal{B}^{(1)},q^{(1)}}\overrightarrow {u}(t,x) =(\mathcal{A}(t,x)\partial_t-\mathcal{B}(t,x)\cdot\nabla_x+\widetilde{\mathcal{A}}(t,x)+q(t,x))\overrightarrow{u}^{(2)}(t,x),&~(t,x)\in Q,\\
\overrightarrow {u}(0,x)=\overrightarrow {0},\qquad \partial_t \overrightarrow {u}(0,x) = \overrightarrow{0}, &~x\in\Omega,\\
\overrightarrow {u}(t,x) = \overrightarrow{0}, &~(t,x)\in  \Sigma.
\end{array}\right.
\end{align}
From the well-posedness result, we have that $\overrightarrow{u}\in \HS^1(Q)$ and $\partial_\nu\overrightarrow{u}\in \LS^2(\Sigma)$. Multiplying the above expression by $\overrightarrow{v}$, where $\overrightarrow{v}$ denotes a solution to the adjoint problem \eqref{Def: adjoint probelm}, and subsequently integrating over $Q$, we arrive at
\begin{align}\label{integral identity 1}
\int\limits_Q &\mathcal{L}_{\mathcal{A}^{(1)},\mathcal{B}^{(1)},q^{(1)}}\overrightarrow {u}(t,x)\cdot\overline{\overrightarrow{v}(t,x)}\,dxdt- \int\limits_Q \overrightarrow {u}(t,x)\cdot\overline{\mathcal{L}^*_{\mathcal{A}^{(1)},\mathcal{B}^{(1)},q^{(1)}}\overrightarrow{v}(t,x)}\,dxdt\nonumber\\
&=\int\limits_\Omega \partial_t\overrightarrow{u}(T,x)\cdot\overline{\overrightarrow{v}(T,x)}\, dx-\int\limits_\Omega\overrightarrow{u}(T,x)\cdot\overline{ \partial_t\overrightarrow{v}(T,x)}\, dx-\int\limits_{\Sigma}\partial_\nu\overrightarrow{u}(t,x)\cdot\overline{\overrightarrow{v}(t,x)}\,dS_xdt.
\end{align}
In the previous expression, we applied the integration by parts formula together with the input–output data specified in \eqref{IBVP for u (defined as differences of solution)}. Subsequently, invoking the relation $\mathcal{L}^*_{\mathcal{A}^{(1)},\mathcal{B}^{(1)},q^{(1)}}\overrightarrow{v}(t,x) = \overrightarrow{0}$, we derive the integral identity \eqref{integral identity for stability}.
\end{proof}
\subsection{Stability estimate for coefficients with first order terms}\label{Stability estimate for matrix potentials}
The stability estimates  \eqref{main result 1} and  \eqref{main result 2} are the primary results, substantiated in this section. The above integral identity motivates us to look for a suitable bound for \begin{align}
\int\limits_Q \big(\mathcal{A}(t,x)\partial_t\overrightarrow{u}^{(2)}(t,x)&-\mathcal{B}(t,x)\cdot\nabla_x\overrightarrow{u}^{(2)}(t,x)+(\widetilde{\mathcal{A}}(t,x)+q(t,x))\overrightarrow{u}^{(2)}(t,x)\big)\cdot\overline{\overrightarrow{v}(t,x)}\,dxdt
\end{align}
in terms of the input-output map.
In view of the notations \eqref{notation used in integral identity}, we first define 
\[
\mathcal{C}:=(\mc{A},\mc{B}),
\]
where its $j^{th}$ component is given by 
\[
\mathcal{C}^{j}:=(A^{(2),j}_0-A^{(1),j}_0,\overrightarrow{A}^{(2),j}-\overrightarrow{A}^{(1),j}), \m{ for }j=1,2,\cdots,m.
\] 
Also, using the Proposition \ref{proposition for GO sol 1} and \ref{proposition for GO sol 2}, we choose 
\begin{align}\label{GO solution for u_2 and v}
\begin{aligned}
\overrightarrow{u_2}(t,x)&=e^{\lambda (t+x\cdot\omega)}\left(\overrightarrow{B}^{(2)}(t,x)+\overrightarrow{R}^{(2)}_{\lambda}(t,x)\right),\\\overrightarrow{v}(t,x)&=e^{-\lambda (t+x\cdot\omega)}\left(\overrightarrow{B}^{(1)}(t,x)+\overrightarrow{R}^{(1)}_{\lambda}(t,x)\right),
\end{aligned}
\end{align}
 where the $j^{th}$ component of $\overrightarrow{B}^{(1)}(t,x)$ and $\overrightarrow{B}^{(2)}(t,x)$ are given by
\begin{align}
&B^{(1),j}(t,x)=K_*^{j}\phi_j(x+t\omega)e^{\int_0^t (1,-\omega)\cdot\mathcal{C}^{(1),j}(s,x+(t-s)\omega)\,ds},\text{ and }\label{jth component of GO solution for v}\\
&B^{(2),j}(t,x)=K^j\phi_j(x+t\omega)e^{-\int_0^t (1,-\omega)\cdot\mathcal{C}^{(2),j}(s,x+(t-s)\omega)\,ds},\label{jth component of GO solution for u2}
\end{align}
respectively, and $K^j,\,K^j_*$ are arbitrary constants. Also, for $i=1,2$, the remainder term $ \overrightarrow{R}^{(i)}_{\lambda}\in \HS^2(Q)$  with
\begin{align}\label{estimate for remainde term 3}
   \lVert \overrightarrow{R}^{(i)}_{\lambda}\rVert_{\HS^k(Q)}\leq C \lambda^{-1+k}\lVert\overrightarrow{\phi}\rVert_{\HS^3(\mathbb{R}^n)},\quad\,\text{for }\, k=0,1,2,
\end{align}
and for sufficiently large $\lambda$. Using the above decay estimates, one can see that for each $1\leq j\leq m,\,i=1,2$ and $r=0,1$,
\begin{align*}
\lVert (1,\omega)\lambda R^{(i),j}_\lambda+\nabla_{t,x}B^{(i),j}+\nabla_{t,x}R^{(i),j}_\lambda\rVert_{H^r(Q)}\leq C\lambda^r\lVert\overrightarrow{\phi}\rVert_{\HS^3(\mathbb{R}^n)}
\end{align*}
holds for $\lambda$ large enough. From the first two terms on the right-hand side of
\eqref{IBVP for u (defined as differences of solution)} in conjunction with \eqref{GO solution for u_2 and v}, we get
\begin{align}\label{estimates for non homogeneous term for u}
\begin{aligned}
\Big(\mathcal{A}&(t,x)\partial_t\overrightarrow{u}^{(2)}(t,x)-\mathcal{B}(t,x)\cdot\nabla_x\overrightarrow{u}^{(2)}(t,x)\Big)\cdot\overline{\overrightarrow{v}(t,x)}\\
&=\sum_{j=1}^m\Big(A^j_0(t,x),\,-\overrightarrow{A}^j(t,x)\Big)\cdot\Big(\big(\lambda,\lambda\omega\big)\big(B^{(2),j}+R^{(2),j}_\lambda\big)+ \nabla_{t,x}B^{(2),j}+\nabla_{t,x}R^{(2),j}_\lambda\Big)\overline{v_j(t,x)}\\
&=\sum_{j=1}^m\lambda\mathcal{C}^{j}(t,x)\cdot\Big(\big(1,-\omega\big)K^jK_*^{j}|\phi_j(x+t\omega)|^2e^{-\int_0^t (1,-\omega)\cdot\mathcal{C}^{j}(s,x+(t-s)\omega)\,ds}\Big)
 +Z_\lambda(t,x),
\end{aligned}
\end{align}
where the remaining terms obtained from the dot product have been clubbed with  $Z_\lambda(t,x)$, and it satisfies the estimate  $\lVert Z_\lambda\rVert_{L^1(Q)}\leq C \lVert\overrightarrow{\phi}\rVert_{\HS^3(\mathbb{R}^n)}$. Moreover, we have
\begin{align}\label{estimates for tilda A and q with u2 and v}
  &\Big\lVert\big(\widetilde{\mathcal{A}}(t,x)\overrightarrow{u}^{(2)}(t,x)+q(t,x)\overrightarrow{u}^{(2)}(t,x)\big)\cdot\overline{\overrightarrow{v}(t,x)}\Big\rVert_{L^1(Q)}\leq C\lVert\overrightarrow{\phi}\rVert^2_{\HS^3(\mathbb{R}^n)}.
\end{align}
We assume that $\o\in \mathbb{S}^{n-1}$ with $|\o-\o_0|<\frac{\varepsilon}{2}$.
From the definition of the input-output map given in \eqref{Def: input-output map} and the notations used in \eqref{notation used in integral identity}, we have
\begin{align*}
    (\Lambda_1-\Lambda_2)(\overrightarrow{\psi_0},\overrightarrow{\psi_1},\overrightarrow{f})=\Big(\overrightarrow{u}(T,\cdot),\,\partial_\nu\overrightarrow{u}\big|_{\Sigma_{-,\varepsilon/2,\og_0}}\Big).
\end{align*}
The trace theorem in conjunction with \eqref{GO solution for u_2 and v} and \eqref{estimate for remainde term 3} leads to
\begin{align*}
    \lVert (\overrightarrow{\psi_0},\overrightarrow{\psi_1},\overrightarrow{f})\rVert_{(\HS^1(\Omega),\LS^2(\Omega),\HS^1(\Sigma))}\leq C e^{\beta\lambda}\|\Lambda_1-\Lambda_2\|\lVert\overrightarrow{\phi}\rVert_{\HS^3(\mathbb{R}^n)},
\end{align*}
where $C,\beta>0$.
Moreover, using the energy estimates given in \eqref{regularity}, we have
\begin{align}\label{continuity of input-output map}
     \|\dd_\nu \us\|_{\LS^2(\Sigma_{-,\varepsilon/2,\og_0})}+\|\us|_{t=T}\|_{\HS^1(\Omega)}\le Ce^{\beta\ld}\|\Lambda_1-\Lambda_2\|\|\overrightarrow{\phi}\|_{\HS^3(\R^n)}.
\end{align}
Next, we bound the first term on the right-hand side of \eqref{integral identity for stability} in terms of the input-output map. Using \eqref{GO solution for u_2 and v}, \eqref{jth component of GO solution for v} and employing the H\"older's inequality and trace theorem, we obtain
\begin{align}\label{boundaryterm}
\begin{aligned}
&\left|\int_{\Omega}\partial_t\us(T,x)\cdot\overline{\vs(T,x)}\ dx\right|\\
&=\left|\int_{\Omega}e^{-\lambda(T+x\cdot\omega)}\partial_t\us(T,x)\cdot(\overrightarrow{B}^{(1)}+\overrightarrow{R}^{(1)}_{\lambda}(T,x))dx\right|\\
&=\Big|\int_{\Omega}e^{-\lambda(T+x\cdot\omega)}\partial_t\us(T,x)\cdot\Big(\big(K_*^{j}\phi_j(x+T\omega)e^{-\int_0^T (1,-\omega)\cdot\mathcal{C}^{(1),j}(s,x+(T-s)\omega)\,ds}\big)_{j=1}^n+\overrightarrow{R}^{(1)}_{\lambda}(T,x)\Big) dx\Big|\\
&\le\frac{C\|\overrightarrow{\phi}\|_{\HS^3(\mathbb{R}^n)}}{\rt\ld}\sqrt{{\int_{\Omega}\ld e^{-2\lambda(T+x\cdot\omega)} 
|\partial_t\us(T,x)|^2}\ dx}.
\end{aligned}
\end{align}
Using the Carleman estimate \eqref{boundary carleman estimate},  we have
\begin{align}\label{RHS of Carleman estimate}
\begin{aligned}
    &{\int_{\Omega}\ld e^{-2\lambda(T+x\cdot\omega)} 
|\partial_t\us(T,x)|^2}\ dx \leq \int_{Q}e^{-2\lambda(t+x\cdot\omega)}|\mathcal{L}_{\mc{A}^{(1)},\mathcal{B}^{(1)},q^{(1)}}\us(t,x)|^2dxdt\\&\quad+\lambda\int\limits_{\Sigma_{-,\omega}}{ e^{-2\lambda(t+x\cdot\omega)}|\omega\cdot \nu(x)||\partial_{\nu}\us|^2}dS_x dt
+\int_{\Omega} e^{-2\lambda(T+x\cdot\omega)}\left(\lambda^3|\us(T,x)|^2+\ld|\nabla_x{\us(T,x)}|^2\right)dx
\end{aligned}
\end{align}
for $\lambda\geq \lambda_0$.
Combining \eqref{GO solution for u_2 and v}, \eqref{IBVP for u (defined as differences of solution)}, \eqref{estimates for non homogeneous term for u} and \eqref{estimates for tilda A and q with u2 and v}, the first term of the right-hand side of the above expression is bounded by
\begin{align}\label{estimate for operator with u}
\int_{Q}e^{-2\lambda(t+x\cdot\omega)}|\mathcal{L}_{\mc{A}^{(1)},\mathcal{B}^{(1)},q^{(1)}}\us(t,x)|^2dxdt \leq C\lambda^2 \|\overrightarrow{\phi}\|^2_{\HS^3(\R^n)}.
\end{align} 
Employing the estimate \eqref{continuity of input-output map} and the continuity of the input-output map defined in \eqref{Def: input-output map},  we estimate the rest of the terms on the right-hand side of \eqref{RHS of Carleman estimate} as 
\begin{align}\label{carleman estimate at time T}
\int_{\Omega} e^{-2\lambda(T+x\cdot\omega)}\left(\lambda^3|\us(T,x)|^2+\ld|\nabla_x{\us(T,x)}|^2\right)dx \leq C e^{\beta\ld}\|\Lambda_1-\Lambda_2\|^2\|\overrightarrow{\phi}\|^2_{\HS^3(\R^n)},
\end{align}
and
\begin{align}\label{estimates on sigma-}
\begin{aligned}
\int\limits_{\Sigma_{-,\omega}}\lambda e^{-2\lambda(t+x\cdot\omega)}|\omega\cdot \nu(x)||\partial_{\nu}\us|^2dS_xdt&\leq C \int\limits_{\Sigma_{-,\varepsilon/2,\og_0}}\lambda e^{-2\lambda(t+x\cdot\omega)}|\omega\cdot \nu(x)||\partial_{\nu}\us|^2dS_xdt,\\
&\leq C e^{\beta\ld}\|\Lambda_1-\Lambda_2\|^2\|\overrightarrow{\phi}\|^2_{\HS^3(\R^n)},
\end{aligned}
\end{align}
where we have used the fact $\Sigma_{-,\omega}\subset\Sigma_{-,\varepsilon/2,\og_0}$.
By combining \eqref{estimate for operator with u}, \eqref{carleman estimate at time T}, and \eqref{estimates on sigma-}, and substituting the resulting bounds into \eqref{RHS of Carleman estimate}, we obtain
\begin{align*}
    {\int_{\Omega}\ld e^{-2\lambda(T+x\cdot\omega)} 
|\partial_t\us(T,x)|^2}\ dx  \leq C (\lambda^2+e^{\beta\ld}\|\Lambda_1-\Lambda_2\|^2)\|\overrightarrow{\phi}\|^2_{\HS^3(\R^n)}.
\end{align*}
Thus we conclude that 
\begin{align}
&\left|\int_{\Omega}\partial_t\us(T,x)\cdot\overline{\vs(T,x)}\ dx\right|
\le C\|\overrightarrow{\phi}\|^2_{\HS^3(\mathbb{R}^n)}\left({\rt\ld} +e^{\beta\lambda}\|\Lambda_1-\Lambda_2\|\right).
\end{align}
By substituting the partial derivatives of the GO solution given in \eqref{GO solution for u_2 and v} into the integral identity below, and proceeding analogously to the derivation of \eqref{boundaryterm}, we obtain
\begin{align}\label{finaltimeestimate}
\begin{aligned}
\left|\int_{\Omega}\us(T,x)\cdot\overline{\partial_t\vs(T,x)}\ dx\right|&\le C\ld\|\overrightarrow{\phi}\|_{\HS^3(\mathbb{R}^n)}\sqrt{{\int_{\Omega}e^{-2\lambda(T+x\cdot\omega)}|\us(T,x)|^2}\ dx},\\
&\le Ce^{\beta\lambda}\|\Lambda_1-\Lambda_2\|\|\overrightarrow{\phi}\|_{\HS^3(\mathbb{R}^n)}^2,
\end{aligned}
\end{align}
and
\begin{align}\label{carleman on boundary sigma}
\begin{aligned}
\left|\int_\Sigma\partial_\nu \us(t,x)\cdot\overline{\vs(t,x)}\ dS_xdt\right|&=\frac{1}{\sqrt\lambda}\ \left|\int_\Sigma\sqrt{\lambda}\partial_\nu \us(t,x)\cdot\overline{\vs(t,x)}\ dS_xdt\right|,\\
&\le \frac{C\|\overrightarrow{\phi}\|_{\HS^3(\mathbb{R}^n)}}{\sqrt\lambda}\sqrt{\int_{\Sigma}{\lambda e^{-2\lambda(t+x\cdot\omega)}}|\dd_\nu \us(t,x)|^2\ dS_xdt}.
\end{aligned}
\end{align}
We now use Carleman estimates to control boundary integrals on the portion $\Sigma_{+,\varepsilon/2,\omega_0}$ in terms of corresponding information on the complementary portion $\Sigma_{-,\varepsilon/2,\omega_0}$. More specifically, since $\Sigma_{+,\varepsilon/2,\omega_0}\subseteq \Sigma_{+,\omega}$, boundary Carleman estimate \eqref{boundary carleman estimate} yields
\begin{align}\label{boundarygreen3}
\begin{aligned}
\int\limits_{\Sigma_{+,\varepsilon/2,\omega_0}}{\ld e^{-2\lambda(t+x\cdot\omega)}}|\partial_\nu \us(t,x)|^2\ dS_xdt&\le \int\limits_{\Sigma_{+,\varepsilon/2,\omega_0}}\frac{\ld}{\varepsilon/2}{ e^{-2\lambda(t+x\cdot\omega)}}|\nu(x)\cdot\omega||\partial_\nu \us(t,x)|^2 dS_xdt,\\
&\le\frac{1}{\varepsilon/2}\int_{\Sigma_{+,\omega}}{\lambda e^{-2\lambda(t+x\cdot\omega)}}|\nu(x)\cdot\omega||\partial_\nu \us(t,x)|^2 dS_xdt,\\
&\le \frac{2C}{\varepsilon}(\lambda^2+ e^{\beta\lambda}\|\Lambda_1-\Lambda_2\|^2 )\|\overrightarrow{\phi}\|_{\HS^3(\mathbb{R}^n)}^2
\end{aligned}
\end{align}
for any $\lambda\geq \lambda_0$.
On combining estimates \eqref{estimates on sigma-}, \eqref{boundaryterm}, \eqref{finaltimeestimate}, \eqref{carleman on boundary sigma} and \eqref{boundarygreen3} into the integral identity \eqref{integral identity for stability}, we get
\begin{align}\label{final estimates with all matrix potential}
\begin{aligned}
\Big|\int\limits_Q \big(\mathcal{A}(t,x)\partial_t\overrightarrow{u}^{(2)}(t,x)&-\mathcal{B}(t,x)\cdot\nabla_x\overrightarrow{u}^{(2)}(t,x)+(\widetilde{\mathcal{A}}(t,x)+q(t,x))\overrightarrow{u}^{(2)}(t,x)\big)\cdot\overline{\overrightarrow{v}(t,x)}\,dxdt\Big|\\
&\leq C\|\overrightarrow{\phi}\|^2_{\HS^3(\mathbb{R}^n)}\left({\rt\ld} +e^{\beta\lambda}\|\Lambda_1-\Lambda_2\| \right)
\end{aligned}
\end{align}
for any $\lambda\geq \lambda_0$. We now estimate the line integrals of the component of the vector potential along the propagation direction of the light rays, using the input–output operator framework. Based on the preceding analysis, we derive the following lemma.
\begin{lemma}\label{first estimate for matrix potential}
For each $j=1,2,\dots,m$, let $\mathcal{C}^{j}=(A^j_0,\overrightarrow{A}^j)\in \CS^\infty_c(Q)$ and scalar valued function $\phi_j\in C_c^\infty(\mathbb{R}^n)$ be the $j^{th}$ component of $\overrightarrow{\phi}$. Assume that  $\omega\in\mathbb{S}^{n-1}$ with $|\omega-\omega_0|<\frac{\varepsilon}{2}$.  Then we can find positive constants $C,\,\beta,\,\lambda_0>0$ such that, the estimate
\begin{align}\label{first stability estimate}
\begin{aligned}
    \Bigg|\int_0^T\int_{\mathbb{R}^{n}}(1,-\omega)\cdot\mathcal{C}^{j}(t,y-t\omega)&e^{-\int^t_0 (1,-\omega)\cdot\mathcal{C}^{j}(s,y-s\omega)\,ds}|\phi_j(y)|^2\,dydt\Bigg|\\
    &\leq C \Bigg(\frac{1}{\sqrt{\lambda}}+e^{\beta\lambda}\lVert\Lambda_1-\Lambda_2\rVert\Bigg)\lVert\overrightarrow{\phi}\rVert^2_{\HS^3(\mathbb{R}^n)},
    \end{aligned}
\end{align}
holds  for all $\lambda\geq\lambda_0>1$.
\end{lemma}
\begin{proof}
To derive \eqref{first stability estimate}, we substitute the  GO solutions into the left-hand side of inequality \eqref{final estimates with all matrix potential}. Fix an arbitrary index $j\in\{1,2,\dots,m\}$, and select the constant vectors $\overrightarrow{K}$ and $\overrightarrow{K_*}$, introduced in \eqref{jth component of GO solution for u2} and \eqref{jth component of GO solution for v}, respectively, by
\begin{equation}\label{choice for constant K}
    \overrightarrow {K}=\left( \delta_{jr} \right)^m_{r=1},\,\,\quad \overrightarrow {K_*}=\left( \delta_{jr}\right)^m_{r=1},
\end{equation}
where $\delta_{jr}$ denotes the Kronecker delta for non-negative integers $j$ and $r$. Dividing \eqref{final estimates with all matrix potential} by $\ld$, and applying \eqref{estimates for non homogeneous term for u} with the same choice of $\overrightarrow{K}$ and $\overrightarrow{K_*}$ as in \eqref{choice for constant K} to the GO solutions defined in \eqref{GO solution for u_2 and v}, we obtain
\begin{align}\label{first stability estimate-01}
\begin{aligned}
\Bigg|\int_Q(1,-\omega)\cdot\mathcal{C}^{j}(t,x)&|\phi_j(x+t\omega)|^2e^{-\int_0^t (1,-\omega)\cdot\mathcal{C}^{j}(s,x+(t-s)\omega)\,ds}\,dxdt\Bigg|\\
&\leq C \Bigg(\frac{1}{\sqrt{\lambda}}+e^{\beta\lambda}\lVert\Lambda_1-\Lambda_2\rVert\Bigg)\lVert\phi\rVert^2_{\HS^3(\mathbb{R}^n)},
\end{aligned}
\end{align}
for each $1\leq j\leq m$ and for $\lambda\geq \lambda_0$. A change of variable $y=x+t\og$ into \eqref{first stability estimate-01} in conjunction with 
 the fact that $\mc{A},\mc{B}\in \mathbf{C}_c^\infty(Q)$, boils down the above expression to 
\begin{align*}
\Bigg|\int_0^T\int_{\mathbb{R}^{n}}(1,-\omega)\cdot\mathcal{C}^{j}(t,y-t\omega)&e^{-\int^t_0 (1,-\omega)\cdot\mathcal{C}^{j}(s,y-s\omega)\,ds}|\phi_j(y)|^2\,dydt\Bigg|\nonumber\\
&\leq C \Bigg(\frac{1}{\sqrt{\lambda}}+e^{\beta\lambda}\lVert\Lambda_1-\Lambda_2\rVert\Bigg)\lVert\overrightarrow{\phi}\rVert^2_{\HS^3(\mathbb{R}^n)},
\end{align*}
for each $1\leq j\leq m$ and for all $\lambda\geq\lambda_0>1$.
\end{proof}

Next, we seek to derive the light ray estimate by substituting particular choices of $\phi_j$. To this end, let $x\in\mb{R}^n$ be arbitrarily fixed, and  $B(0,1)$ denotes the open unit ball in $\R^n$. 
Now, we choose 
 $\phi_j\in C_c^{\iy}(B(0,1))$ with $\|\phi_j\|_{L^2(\mathbb{R}^n)}=1$ and for $h>0$, we define
\[ \phi_{j,h}(y):=h^{-\frac{n}{2}}\phi_j\Big(\frac{y-x}{h}\Big),\qquad y\in\mathbb{R}^n, \] 
for each $j=1,2,\dots,m.$ As a result, $\phi_{j,h}\in C_c^{\iy}(\R^n)$ with 
$\supp(\phi_{j,h})\subseteq B(x,h)$ and $\|\phi_{j,h}\|_{L^2(\mathbb{R}^n)}=1$. Moreover, there exists a constant $C>0$ depending on $\phi_j$ such that $\|\phi_{j,h}\|_{H^3(\mathbb{R}^n)}\le Ch^{-3}$. On substituting these mollified function into the inequality in \eqref{first stability estimate}, we get 
\begin{align}\label{estimates with mollification}
\begin{aligned}
    \Bigg|\int_0^T\int_{\mathbb{R}^{n}}(1,-\omega)\cdot\mathcal{C}^{j}(t,y-t\omega)&e^{-\int^t_0 (1,-\omega)\cdot\mathcal{C}^{j}(s,y-s\omega)\,ds}|\phi_{j,h}(y)|^2\,dydt\Bigg|\\
    &\leq C \Bigg(\frac{1}{\sqrt{\lambda}}+e^{\beta\lambda}\lVert\Lambda_1-\Lambda_2\rVert\Bigg)\lVert\overrightarrow{\phi_{h}}\rVert^2_{\HS^3(\mathbb{R}^n)},
    \end{aligned}
\end{align}
where the function $\overrightarrow{\phi_h}(y)=(\phi_{1,h}(y), \phi_{2,h}(y),\cdots,\phi_{m,h}(y))^T$ and $|\omega-\omega_0|<\frac{\varepsilon}{2}$. Further, we observe that the first term in the above expression can be written as 
\begin{align}\label{Bound: 5.25}
\begin{aligned}
    \int_0^T\int_{\mathbb{R}^{n}}&(1,-\omega)\cdot\mathcal{C}^{j}(t,y-t\omega)e^{-\int^t_0 (1,-\omega)\cdot\mathcal{C}^{j}(s,y-s\omega)\,ds}|\phi_{j,h}(y)|^2\,dydt\\&=-\int_{\mathbb{R}^n} \int_{0}^{T}
    \partial_t{\left( e^{-\int_{0}^{t}{(1,-\omega)\cdot\mathcal{C}^{j}(s,y-s\omega)ds}}\right)}|\phi_{j,h}(y)|^2\ dtdy\\&=-\int_{\mathbb{R}^n}   \left( e^{-\int_{0}^{T}{(1,-\omega)\cdot\mathcal{C}^{j}(s,y-s\omega)ds}}-1\right)|\phi_{j,h}(y)|^2~dy.
    \end{aligned}
\end{align}
Next, we combine the previous expression with  \eqref{estimates with mollification} to get
\begin{equation}
\left|\int_{\R^n}|\phi_{j,h}(y)|^2\left( e^{-\int_{0}^{T}{(1,-\omega)\cdot\mathcal{C}^{j}(s,y-s\omega)ds}}-1\right)dy\right|\le Ch^{-6}\left(\frac{1}{\sqrt{\lambda}}+e^{\beta\lambda}\|\Lambda_1-\Lambda_2\|\right).\label{approximation1}
\end{equation}
Further, for arbitrary $K>0$, we define a continuous function $g:[-K,K]\rightarrow \mathbb{R}$ by 
\begin{align*}
g(x)=\begin {cases}
\frac{e^x-1}{x},\quad &\m{for }x\neq 0,\\
1,\quad\quad &\m{for }x=0.
\end{cases}
\end{align*}		
By the extreme-value theorem, a continuous function on a compact interval must attain both its maximum and minimum. However, $g$ is non-negative. Thus, there exists a constant $C>0$ which depends on $M$ such that 
\begin{align}\label{Def: Lower bound for exp} 
|x|\le C|e^x-1|,\ \text{ for all } x\in[-M,M]. \end{align}

Using \eqref{approximation1} and \eqref{Def: Lower bound for exp} in conjunction with mean value theorem, we prove the following light ray transform.
\begin{lemma}\label{existence of delta lambda0}
Suppose that the assumptions given in Lemma \ref{first estimate for matrix potential} hold. Then there exist $\delta,\,\ld_0,$ and $C>0$ such that for each $x\in\mb{R}^n$ and  $\lambda\geq\lambda_0>1$, the following estimate
\begin{align}\label{second stability estimate:01} 
\left|\int_{\mathbb{R}}{(1,-\omega)\cdot\mathcal{C}^{j}(s,x-s\omega)\ ds}\right|\le C\left(\frac{1}{\lambda^\delta}+\rm e^{\beta\lambda}\|\Lambda_1-\Lambda_2\|\right),
\end{align}
 holds.
\end{lemma}	
\begin{proof}
By invoking the estimate in \eqref{Bound: 5.25}, it follows that
\begin{align}
&\left|e^{-\int_{0}^{T}(1,-\og)\cdot\mathcal{C}^{j}(s,x-s\og)ds}-1\right|=\left|\int_{\R^n}|\phi_{j,h}(y)|^2(e^{-\int_{0}^{T}-(1,-\og)\cdot\mathcal{C}^{j}(s,x-s\og)ds}-1)\ dy\right|\\
&\le\  \left|\int_{\R^n}|\phi_{j,h}(y)|^2(e^{-\int_{0}^{T}(1,-\og)\cdot\mathcal{C}^{j}(s,x-s\og)ds}-e^{-\int_{0}^{T}(1,-\og)\cdot\mathcal{C}^{j}(s,y-s\og)ds})\ dy\right|\\
&\qquad+\left|\int_{\R^n}|\phi_{j,h}(y)|^2(e^{-\int_{0}^{T}(1,-\og)\cdot\mathcal{C}^{j}(s,y-s\og)ds}-1)\ dy\right|.
\end{align}
Now  to estimate the first term on the right-hand side of the previous expression, we implement the mean value theorem twice to get
\begin{align}
 \big|e^{-\int_{0}^{T}{(1,-\omega)\cdot\mathcal{C}^{j}(s,x-s\omega)ds}}&- e^{-\int_{0}^{T}{(1,-\omega)\cdot\mathcal{C}^{j}(s,y-s\omega)ds}}\big|&\\
&\le C\left|\int_{0}^{T}{(1,-\omega)\cdot\mathcal{C}^{j}(s,x-s\omega)ds}-\int_{0}^{T}{(1,-\omega)\cdot\mathcal{C}^{j}(s,y-s\omega)ds}\right|\\&\le C|x-y|,
\end{align}
where $C>0$ depends on $Q$ and apriori bounds of $\mathcal{C}^{j}$. 
Combining the previous calculation with inequality \eqref{approximation1}, results into
\begin{align}
\left|e^{-\int_{0}^{T}(1,-\og)\cdot\mathcal{C}^{j}(s,x-s\og)ds}-1\right|
\le  C\left(h+h^{-6}\Big(\frac{1}{\sqrt{\lambda}}+ e^{\beta\lambda}\|\Lambda_1-\Lambda_2\|\Big)\right), 
\end{align}
where $|\og-\og_0|<\frac{\varepsilon}{2}$. 
Here we have also used that $\supp(\phi_{j,h})\subseteq B(x,h)$. Armed with \eqref{Def: Lower bound for exp}, we then conclude that 
\begin{align}\label{integral in CJ and exponential}
\begin{aligned}
\left|\int_{0}^{T}(1,-\og)\cdot\mathcal{C}^{j}(s,x-s\og)ds\right|&\le C\left| e^{-\int_{0}^{T}(1,-\og)\cdot\mathcal{C}^{j}(s,x-s\og)ds}-1\right|\\&\le  C\left(h+h^{-6}\Big(\frac{1}{\sqrt{\lambda}}+ e^{\beta\lambda}\|\Lambda_1-\Lambda_2\|\Big)\right).
\end{aligned}
\end{align}
Estimate \eqref{second stability estimate:01} can be accomplished by choosing sufficiently small $h > 0$ such that $h=\ld^{-\frac{1}{14}}$ with $\ld\geq\ld_0>1$.
Thus, there exists $\delta>0$ such that
\begin{align*}
\left|\int_{\mathbb{R}}{(1,-\omega)\cdot\mathcal{C}^{j}(s,x-s\omega)ds}\right|=\left|\int_{0}^{T}{(1,-\omega)\cdot\mathcal{C}^{j}(s,x-s\omega)ds}\right|\le C\left(\frac{1}{\lambda^\delta}+ e^{\beta\lambda}\|\Lambda_1-\Lambda_2\|\right),
\end{align*} 
where $|\og-\og_0|<\frac{\varepsilon}{2}$ and  $\lambda\ge\lambda_0>1$.
\end{proof}
As a corollary, we establish inequality \eqref{Fourier estimates along ray}, formulated in terms of the space-time Fourier transform. To this end, we integrate the componentwise estimates provided by \eqref{second stability estimate:01}, which yields the desired Fourier estimate. More precisely, for $\og\in\mb{S}^{n-1}$ such that $\lvert\og-\og_0\rvert<\varepsilon/2$, the space-time Fourier transform  \eqref{defn:Fourier transform} yields
\begin{align*}
\bigl|(1,-\og)\cdot\widehat{\mathcal{C}^{j}}(\og\cdot\xi,\xi)\bigr|
&=\left|\int_{\R^{1+n}}(1,-\og)\cdot\mathcal{C}^{j}(s,y)\,e^{-i(s(\og\cdot\xi)+y\cdot \xi)}\ ds\,dy\right|\\
&=\left|\int_{\R^n}e^{-iz\cdot\xi}\left(\int_{\R}(1,-\og)\cdot {\mathcal{C}^{j}}(s,z-\og s)\,ds\right)\ dz\right|,
\end{align*}
where we used the change of variables $y=z-s\og$. Since $|e^{-iz\cdot\xi}|=1$ and $\mathcal{C}^{j}\in \mathbf{C}_c^{\infty}(Q)$, estimate \eqref{second stability estimate:01} implies
\begin{align}\label{5.31}
\begin{aligned}
\bigl|(1,-\og)\cdot\widehat{\mathcal{C}^{j}}(\og\cdot\xi,\xi)\bigr|
&\le \int_{\R^n}\left|\int_{\R}(1,-\og)\cdot\mathcal{C}^{j}(s,z-s\og)\,ds\right|\ dz \\
&\le C\left(\frac{1}{\lambda^\delta}+e^{\beta\lambda}\|\Lambda_1-\Lambda_2\|\right).
\end{aligned}
\end{align}
Using the ideas of \cite{bellassoued2019stability} and employing Lemma \ref{existence of delta lambda0}, we establish Fourier estimates along light rays.
\begin{corollary}\label{existence of a cone}
Assume that the hypotheses of Lemma \ref{first estimate for matrix potential} are satisfied. Then there exists an open cone $\mc{S}\subset \mb{R}^{1+n}$ such that, for every $(\tau,\xi)\in \mc{S}$, one has
\begin{align}\label{Fourier estimates along ray}
\big|(1,-\og(\tau,\xi))\cdot\widehat{\mathcal{C}^{j}}(\tau,\xi)\big|\le C\left(\frac{1}{\lambda^\delta}+ e^{\beta\lambda}\|\Lambda_1-\Lambda_2\|\right),
\end{align}
for all $\og(\tau,\xi)\in \mb{S}^{n-1}$ satisfying $|\og-\og_0|<\frac{\varepsilon}{2}$ and the constraint $\og(\tau,\xi)\cdot\xi=\tau$.
\end{corollary}  
\begin{proof}
Armed with \eqref{5.31}, we next identify a subset of $\R^{1+n}$ on which this estimate \eqref{Fourier estimates along ray} holds. Define
\[
M_{\varepsilon}=\displaystyle \cup_{|\og-\og_0|<{\varepsilon/2}}\og^\perp\m{ and } N_{\varepsilon}=\{(\tau,\xi)\in\R\times M_{\varepsilon/2};\ |\tau|<\tfrac{\varepsilon}{8}|\xi|\}.
\]
Since the interior of $M_\varepsilon$ is nonempty, one may select an open cone $\mc{S}\subset N_{\varepsilon}$. Assume $(\tau_0,\xi_0)\in \mc{S}$. Then there exists $\og_1\in\mb{S}^{n-1}$ with $|\og_1-\og_0|<\frac{\varepsilon}{4}$ such that $\xi_0\cdot\og_1=0$. We then define
\[
\og=\frac{\tau_0}{|\xi_0|^2}\xi_0+\rt{1-\frac{\tau_0^2}{|\xi_0|^2}}\og_1,
\]
so that $\tau_0=\xi_0\cdot \og$. Moreover, using $|\tau_0|<\frac{\varepsilon}{8}|\xi_0|$ (since $(\tau_0,\xi_0)\in\mc{S}\subset N_\varepsilon$) and $|\og_1-\og_0|<\frac{\varepsilon}{4}$, we obtain
\begin{align}
|\og-\og_0|
&\le\left|\frac{\tau_0}{|\xi_0|^2}\right||\xi_0|
+\rt{1-\frac{\tau_0^2}{|\xi_0|^2}}|\og_1-\og_0|
+\left|\left(\rt{1-\frac{\tau_0^2}{|\xi_0|^2}}-1\right)\og_0\right|\nn\\
&\le\frac{\varepsilon}{8}+\frac{\varepsilon}{4}+\frac{\varepsilon}{8}
=\frac{\varepsilon}{2}.\nn
\end{align}
Because $(\tau_0,\xi_0)\in\mc{S}$ was arbitrary, it follows that for every $(\tau,\xi)\in\mc{S}$ there exists $\og\in\mb{S}^{n-1}$ satisfying $|\og-\og_0|<\varepsilon/2$ and $\tau=\og\cdot\xi$, and consequently
\[
\bigl|(1,-\og)\cdot\widehat{\mc{C}^{j}}(\tau,\xi)\bigr|
\le C\left(\frac{1}{\lambda^\delta}+ e^{\beta\lambda}\|\Lambda_1-\Lambda_2\|\right).
\]
Hence the result follows.
\end{proof}
Since 
$
\div_{(t,x)}\big(\mc{A}^{(1)},\mc{B}^{(1)}\big)= \div_{(t,x)}\big(\mc{A}^{(2)},\mc{B}^{(2)}\big),
$
application of the space–time Fourier transform \eqref{defn:Fourier transform} implies the transversality condition
\[
(\tau,\xi)\cdot\widehat{\mc{C}^{j}}(\tau,\xi)=0.
\]
We next derive uniform bounds for the vector potential on a cone by exploiting this divergence-free constraint. In particular, we invoke the following estimates established in \cite{senapati2020stability}.

\begin{lemma}\cite{senapati2020stability}\label{set of equation and bound for G}
Fix $(\tau,\xi)\in\mb{R}^{1+n}$ such that $|\tau|< \frac{1}{2}|\xi|$, and consider the system
\begin{align}\label{unit vector}
\begin{aligned}
(1,-\omega(\tau,\xi))\cdot\widehat{\mc{C}^{j}}(\tau,\xi)&= G(\xi,\omega(\tau,\xi)),\\
(\tau,\xi)\cdot\widehat{\mc{C}^{j}}(\tau,\xi)&=0, \\
\m{ where } \omega(\tau,\xi)\in\mb{S}^{n-1} &\m{ satisfies } \omega(\tau,\xi)\cdot\xi=\tau.
\end{aligned}
\end{align}
Then there exists a constant $C>0$, independent of $\{(\tau,\xi)\in\mb{S}^{n}:|\tau|<\frac{1}{2}|\xi|\}$, such that for a suitable selection of vectors $\omega^k(\tau,\xi)$ satisfying \eqref{unit vector}, one has
\begin{align*}
\widehat{C^j}_i(\tau,\xi)&=\sum_{k=1}^{n}c_{k,i}(\tau,\xi)G(\xi,\omega^k(\tau,\xi)),\quad\text{ for all } i\in\{0,\ldots,n\},
\end{align*}
where
$
c_{k,i}(\tau,\xi)=\frac{1}{\det M_{(\tau,\xi)}}\,C_{i,k}(\tau,\xi),
$
and $C_{i,k}(\tau,\xi)$ denotes the $(i,k)^{\mathrm{th}}$ cofactor of the matrix $M_{(\tau,\xi)}$ given by 
\begin{align*}
    M_{(\tau,\xi)}=\begin{pmatrix}
1 & -\og_1^1(\tau,\xi) & \cdots & -\og_1^n(\tau,\xi) \\
1 & -\og_2^1(\tau,\xi) & \cdots & -\og_2^n(\tau,\xi) \\
\vdots  & \vdots  & \ddots & \vdots  \\
1 & -\og_n^1(\tau,\xi) & \cdots & -\og_n^n(\tau,\xi) \\
\frac{\tau}{\rt{{\tau}^2+|\xi|^2}} & \frac{\xi^1}{\rt{{\tau}^2+|\xi|^2}} & \cdots & \frac{\xi^n}{\rt{{\tau}^2+|\xi|^2}} 
\end{pmatrix}.
\end{align*}
Moreover, we have
\[
|\widehat{C^{j}}_i(\tau,\xi)|\le C\max\limits_{1\le k\le n}\big|G(\xi,\omega^k(\tau,\xi))\big|, \quad \text{for all } i\in\{0,1,\dots,n\}.
\]
\end{lemma}

To control the Fourier transform of the vector potentials on large balls, we subsequently employ the Vessella’s analytic continuation result. As a preliminary step, we first establish Fourier estimates uniform on a cone and, in particular, independent of individual light-ray directions; for this purpose, we follow the approach developed in \cite{salazar2014stability}.
Moreover, when $n=2$, punctured spherical caps are disconnected, which accounts for the failure of the partial data strategy in that dimension. Consequently, throughout this argument we restrict attention to dimensions $n\ge 3$. Let $(\tau,\xi)\in\mb{R}^{1+n}$ satisfy $|\tau|<\frac{\varepsilon}{8}|\xi|$, and define the subset $\mc{B}_{\tau,\xi}\subset\mb{S}^{n-1}$ by
\[
\mc{B}_{\tau,\xi}:=\{\omega\in\mb{S}^{n-1}:\ \omega\cdot\xi=\tau \ \mbox{and}\ |\omega-\omega_0|<\varepsilon\}.
\]
The following geometric property, proved in \cite{senapati2020stability}, will be used repeatedly.

\begin{lemma}\cite{senapati2020stability}\label{existence of independent vectors}
Assume $n\geq 3$ and that $\mc{B}_{\tau,\xi}$ is nonempty. Then $\mc{B}_{\tau,\xi}$ contains $n$ linearly independent vectors.
\end{lemma}

Employing Lemma \ref{existence of independent vectors}, we denote the set of $n$ linearly independent vectors of $\mc{B}_{\tau_0,\xi_0}$ by $\{\omega_k(\tau_0,\xi_0)\}_{1\le k\le n}\subset\mathbb{S}^{n-1}$, where $(\tau_0,\xi_0)\in \mc{S}\cap\mb{S}^n$. In addition, invoking Corollary \ref{existence of a cone}, we have $(\tau_0,\xi_0)\cdot(1,-\omega_k(\tau_0,\xi_0))=0,$ for each $k=1,2,\dots,n,$ and the set $\{(1,-\omega_k(\tau_0,\xi_0))\}_{1\le k\le n}$ is linearly independent in $\mathbb{R}^{1+n}$. Therefore the matrix $M_{(\tau_0,\xi_0)}$ given by
    \begin{align*}
M_{(\tau_0,\xi_0)}=\begin{pmatrix}
1 & -\og_1^1(\tau_0,\xi_0) & \cdots & -\og_1^n(\tau_0,\xi_0) \\
1 & -\og_2^1(\tau_0,\xi_0) & \cdots & -\og_2^n(\tau_0,\xi_0) \\
\vdots  & \vdots  & \ddots & \vdots  \\
1 & -\og_n^1(\tau_0,\xi_0) & \cdots & -\og_n^n(\tau_0,\xi_0) \\
\frac{\tau_0}{\rt{{\tau_0}^2+|\xi_0|^2}} & \frac{\xi_0^1}{\rt{{\tau_0}^2+|\xi_0|^2}} & \cdots & \frac{\xi_0^n}{\rt{{\tau_0}^2+|\xi_0|^2}} 
\end{pmatrix} 	
\end{align*} 
    is invertible, where $\xi_0=(\xi_0^1, \xi_0^2,\cdots,\xi_0^n)$. The conclusion of Corollary \ref{existence of a cone} can be strengthened by establishing a Fourier-transform estimate that is independent of light-ray directions in an open cone subset of $\mathcal{S}$. To do so, we invoke Lemma \ref{existence of independent vectors}, which allows us to construct an invertible matrix from a collection of linearly independent vectors. The non-vanishing determinant of this matrix is essential for guaranteeing the existence of an open cone. In particular, we establish the following lemma.
\begin{lemma}\label{matrix potential estimate in cone}
Let $\mc{S}$ be the open cone as described in Corollary \ref{existence of a cone}. Then there exists an open cone $\mc{S}_0\subseteq\mc{S}$ such that the following holds
\[|\widehat{\mc{C}^{j}}(\tau,\xi)|\le C\left(\frac{1}{\lambda^\delta}+ e^{\beta\lambda}\|\Lambda_1-\Lambda_2\|\right),\quad \m{for }(\tau,\xi)\in\mc{S}_0.\]
\end{lemma}
\begin{proof}
Let $\tilde{\mc{S}}$ denote the set of parameters $(\tau,\xi)\in\mathbb{S}^{n}$ contained in a sufficiently small neighborhood of $(\tau_0,\xi_0)$ such that the hyperplane
$
\{x\in\mathbb{R}^n;\ \xi\cdot x=\tau\}
$
depends continuously on $(\tau,\xi)$. More precisely, as $(\tau,\xi)$ varies continuously in $\tilde{\mc{S}}$, one may choose a family of $n$ linearly independent vectors $\{\omega_k(\tau,\xi)\}_{1\le k\le n}\subset \mc{B}_{\tau,\xi}$ depending continuously on $(\tau,\xi)$. Consequently, the matrix $M_{(\tau,\xi)}$ varies continuously with $(\tau,\xi)$ and retains its non-singular behaviors on $\tilde{\mc{S}}$, i.e. $\det M_{(\tau,\xi)}\neq 0$ for all $(\tau,\xi)\in\tilde{\mc{S}}$.
Moreover, due to the extreme-value theorem, there exists a constant  $c>0$ such that
\[
0<c\le |\det M_{(\tau,\xi)}|,\quad \text{for all }(\tau,\xi)\in\tilde{\tilde{\mc{S}}}\Subset \tilde{\mc{S}}.
\]
For $(\tau,\xi)\in\tilde{\tilde{\mc{S}}}$, consider a linear system
\begin{align}
\label{system1}(1,-\og_k(\tau,\xi))\cdot \widehat{\mc{C}^{j}}(\tau,\xi)&=G(\xi,\omega_k(\tau,\xi)),\quad\m{ for }k=1,\ldots,n,\\
\frac{1}{\sqrt{\tau^2+\xi^2}}(\tau,\xi)\cdot \widehat{\mc{C}^{j}}(\tau,\xi)&=0.
\end{align}
We can write the above system in matrix form as $M_{(\tau,\xi)} X=B$, where 
\begin{align*}
M_{(\tau,\xi)}=\begin{pmatrix}
1 & -\og_1^1(\tau,\xi) & \cdots & -\og_1^n(\tau,\xi) \\
1 & -\og_2^1(\tau,\xi) & \cdots & -\og_2^n(\tau,\xi) \\
\vdots  & \vdots  & \ddots & \vdots  \\
1 & -\og_n^1(\tau,\xi) & \cdots & -\og_n^n(\tau,\xi) \\
\frac{\tau}{\rt{{\tau}^2+|\xi|^2}} & \frac{\xi^1}{\rt{{\tau}^2+|\xi|^2}} & \cdots & \frac{\xi^n}{\rt{{\tau}^2+|\xi|^2}} 
\end{pmatrix},\quad X=\begin{pmatrix}
    \widehat{C^j}_0\\\widehat{C^j}_1\\ \vdots\\\widehat{C^j}_n
\end{pmatrix},\quad B=\begin{pmatrix}
    G(\xi, \o_1\\G(\xi, \o_1)\\ \vdots\\ G(\xi, \o_n)\\0
\end{pmatrix}.
\end{align*} 
Corollary \ref{existence of a cone} provides an upper bound for $G$ on the cone $\mc{S}$. Also,  $M_{(\tau,\xi)}$ is a homogeneous matrix of degree zero that is invertible on $\tilde{\tilde{\mc{S}}}$. Combining these properties, one can transfer local estimates from $\tilde{\tilde{\mc{S}}}$ to corresponding estimates on the open cone in $\mathbb{R}^{1+n}$ generated by $\tilde{\tilde{\mc{S}}}$.
In particular, by Lemma \ref{set of equation and bound for G}, for $(\tau,\xi)\in\tilde{\tilde{\mc{S}}}$ and $r>0$ we obtain
\begin{align*}
\widehat{C^j}_i(r\tau,r\xi)&=\sum_{k=1}^{n}c_{k,i}(r\tau,r\xi)G(r\xi,\omega^k(r\tau,r\xi)),\quad i\in\{0,\ldots,n\}\\
&=\sum_{k=1}^{n}c_{k,i}(\tau,\xi)G(r\xi,\omega^k(\tau,\xi)),
\end{align*}
where
$
c_{k,i}(\tau,\xi)=\frac{1}{\det M(\tau,\xi)}\,C_{i,k}(\tau,\xi),
$
and $C_{i,k}(\tau,\xi)$ denotes the $(i,k)^{\mathrm{th}}$ cofactor of $M_{(\tau,\xi)}$. Now for the open cone $\mc{S}_0\ (\equiv\cup_{r>0}r\tilde{\tilde{\mc{S}}})$ in $\mathbb{R}^{1+n}$, Corollary \ref{existence of a cone} along with previous calculation ensures the existence of $C>0$ such that the estimate
\begin{align}
|\widehat{C^j}_i(\tau,\xi)|\le C\left(\frac{1}{\lambda^\delta}+ e^{\beta\lambda}\|\Lambda_1-\Lambda_2\|\right),\label{oncone}
\end{align}
holds for all $i\in\{0,\ldots,n\}$.
This completes the proof.
\end{proof}
As discussed earlier, Vessella's analytic continuation argument (refer to \cite{vessella1999continuous}) is now required to derive Fourier estimates of vector potentials over arbitrarily large balls.
To this end, we define an analytic function \[f_k(t,x)= 
\widehat{\mc{C}^{j}}(kt,kx)\quad\m{ for }k>0\m{ and }(t,x)\in\R^{1+n}.\]   
Indeed, for any multi-index $\gamma$, an application of the Fourier transform leads to
\begin{align}      
\quad|\partial_{(t,x)}^{\gamma}f_k(t,x)|&=\ |\partial_{(t,x)}^{\gamma}\widehat{\mc{C}^{j}}(kt,kx)|,\\
&=\left|\int_{\R^{1+n}}e^{-\mathrm{i}k(s,y)\cdot (t,x)}(-\mathrm{i}k)^{|\gamma|}(s^2+|y|^2)^{\frac{|\gamma|}{2}}\mc{C}^{j}(s,y)\ dsdy\right| \\
&\le(2T^2)^{\frac{|\gamma|}{2}}k^{|\gamma|}\int_{\R^{1+n}}|\mc{C}^{j}(s,y)|\ dsdy  \\
&\le C\ T^{|\gamma|}\ k^{|\gamma|}
,
\end{align} 
where we have used $ T>\diam(\Omega)$ and $\mc{C}^{j}\in C_c^\infty(\mathbb{R}^{1+n})$. On further simplification, the previous expression boils down to
\begin{align}
&\quad|\partial_{(t,x)}^{\gamma}f_k(t,x)|\le C \ T^{|\gamma|}k^{|\gamma|}\frac{|\gamma|!}{|\gamma|!}\leq  C \ T^{|\gamma|}e^{k}|\gamma|!\ \m{ for }(t,x)\in\R^{1+n}\m{ and multi-index }\gamma. \label{prevessellla1}  
\end{align}
Since $f_k$ satisfies \eqref{prevessellla1}, we may invoke the conditional stability estimate of Vessella \cite{vessella1999continuous} for $f_k$. In particular, there exists $\theta\in(0,1)$ such that
\begin{align}
\|f_k\|_{\LS^\infty(B(0,1))}\le C e^{k(1-\theta)}\,\|f_k\|_{\LS^\infty(\mc{S}_0\cap B(0,1))}^{\theta}.
\end{align}
Substituting $f_k(t,x)=\widehat{\mc{C}^{j}}(kt,kx)$ into the preceding inequality and applying Lemma \ref{matrix potential estimate in cone}, we obtain
\begin{align}
\|\widehat{\mc{C}^{j}}\|_{\LS^\infty(B(0,k))}\le C e^{k(1-\theta)}\left(\frac{1}{\lambda^\delta}+ e^{\beta\lambda}\|\Lambda_1-\Lambda_2\|\right)^{\theta}.
\label{vassella2}
\end{align}
We next derive an upper bound for $\mc{C}^{j}$ in the Sobolev space $\HS^{-1}(\R^{1+n})$ in terms of the input--output operator. This estimate, combined with the Sobolev embedding theorem, yields the desired stability bound for $(\mathcal{A},\mathcal{B})$ which is in a much stronger Sobolev norm. Specifically, using \eqref{vassella2} together with a priori assumptions on the potentials, we estimate $\|\mc{C}^{j}\|_{\HS^{-1}(\R^{1+n})}$ in terms of the norm of the input--output operator.
\begin{align}
    \|\mc{C}^{j}\|_{\HS^{-1}(\R^{1+n})}^{\frac{2}{\theta}}=\left(\int_{\R^{1+n}}(1+s^2+|y|^2)^{-1}|\widehat{\mc{C}^{j}}(s,y)|^2dsdy\right)^{\frac{1}{\theta}}.
\end{align}
We split the inner integral over the disjoint sets $B(0,k)\cup B(0,k)^c$ and exploit the elementary bound $(1+s^2+\lvert y\rvert^2)^{-1}\leq 1$. Moreover, for any $(s,y)\in B(0,k)^c$, one has the lower estimate $s^2+\lvert y\rvert^2\geq k^2$. Finally, invoking the convexity of the function $h(x)=x_{+}^{1/\theta}$ for $\theta\in(0,1)$, together with \eqref{vassella2}, yields
\begin{align}\label{last1}    
\begin{aligned}
\|\mc{C}^{j}\|_{\HS^{-1}(\R^{1+n})}^{\frac{2}{\theta}}&=\left(\int\limits_{B(0,k)}(1+s^2+|y|^2)^{-1}|\widehat{\mc{C}^{j}}(s,y)|^2dsdy+\int\limits_{B(0,k)^c}(1+s^2+|y|^2)^{-1}|\widehat{\mc{C}^{j}}(s,y)|^2dsdy\right)^{\frac{1}{\theta}}\\
&\le C\left(k^{n+1}\|\widehat{\mc{C}^{j}}\|_{L^\infty(B(0,k))}^2+\frac{1}{k^2}\right)^{\frac{1}{\theta}}\\
&\le C\left(k^{\frac{n+1}{\theta}}\|\widehat{\mc{C}^{j}}\|_{L^\infty(B(0,k))}^{\frac{2}{\theta}}+\frac{1}{k^{\frac{2}{\theta}}}\right)\\
&\le C\left( k^{\frac{n+1}{\theta}}e^{\frac{2k(1-\theta)}{\theta}}\left(\frac{1}{\lambda^{2\delta}}+ e^{\beta\lambda}\|\Lambda_1-\Lambda_2\|^2\right)+\frac{1}{k^{\frac{2}{\theta}}}\right)\\ 
&\le C\left(\underset{I}{\underbrace{\frac{k^{\frac{n+1}{\theta}}e^{\frac{2k(1-\theta)}{\theta}}}{\ld^{2\delta}}}}
+\underset{II}{\underbrace{k^{\frac{n+1}{\theta}}e^{\beta\ld+\frac{2k(1-\theta)}{\theta}}\|\Lambda_1-\Lambda_2\|^2}}+\underset{III}{\underbrace{\frac{1}{k^{\frac{2}{\theta}}}}}\right).
\end{aligned}
\end{align}
Now we choose $k$ sufficiently large and \[\ld=\exp\left({\dfrac{k(1-\theta)}{\theta\delta}}\right)k^{\frac{n+3}{2\theta\delta}}\] so that $I$ and $III$ are comparable. With this choice of $\lambda$, $II$ reduces to
\[{k^{\frac{n+1}{\theta}}\exp\left({\beta\left(\exp\left({\frac{k(1-\theta)}{\theta\delta}}k^{\frac{n+3}{2\theta\delta}}\right)\right)+\frac{2k(1-\theta)}{\theta}}\right)\|\Lambda_1-\Lambda_2\|^2}.\]
Hence there exists $\varsigma>0$, independent of $k$ such that 
\begin{align}\label{kchoice}
k^{\frac{n+1}{\theta}}\exp\left({\beta\left(\exp\left({\frac{k(1-\theta)}{\theta\delta}}k^{\frac{n+3}{2\theta\delta}}\right)\right)+\frac{2k(1-\theta)}{\theta}}\right)\|\Lambda_1-\Lambda_2\|^2\le C e^{e^{\varsigma k}}\|\Lambda_1-\Lambda_2\|^2.
\end{align}
	
 Choosing $k=\frac{1}{\varsigma}\log\big|\log\|\Lambda_1-\Lambda_2\|\big|$ (provided \(\|\Lambda_1 - \Lambda_2\| \) is small enough) will leads to
\begin{align}    
&\|\mc{C}^{j}\|_{\HS^{-1}(\mathbb{R}^{1+n})}^{\frac{2}{\theta}} \le C\left(\|\Lambda_1-\Lambda_2\|
+\left(\log|\log\|\Lambda_1-\Lambda_2\||\right)^{-\frac{2}{\theta}}\right).
\end{align}    
Summing over the indices $1\leq j\leq m$, we obtain
\begin{align}\label{last0}  \|\mc{C}\|_{\HS^{-1}(\mathbb{R}^{1+n})} \le C\left(\|\Lambda_1-\Lambda_2\|^{\frac{\theta}{2}}+\big|\log|\log\|\Lambda_1-\Lambda_2\||\big|^{-1}\right).
\end{align}
However, for \(\|\Lambda_1 - \Lambda_2\|\geq c \), where $c>0$ is fixed, we can proceed with the following steps.
\begin{align}    
\|\mc{C}^{j}\|_{\HS^{-1}(\mathbb{R}^{1+n})}^{\frac{2}{\theta}} \le C\|\mc{C}^{j}\|_{\LS^\infty(\mathbb{R}^{1+n})}^{\frac{2}{\theta}}\le \frac{C}{c}{c}&\le \frac{C}{c}\|\Lambda_1-\Lambda_2\|\ .
\end{align}
Again, summing over the indices $1\leq j\leq m$ to get
\begin{align}\label{last2}   
\|\mc{C}\|_{\HS^{-1}(\mathbb{R}^{1+n})}^{\frac{2}{\theta}} \le C\|\mc{C}\|_{\LS^\infty(\mathbb{R}^{1+n})}^{\frac{2}{\theta}}\le \frac{C}{c}\|\Lambda_1-\Lambda_2\|\ .
\end{align}
So, form both \eqref{last0} and \eqref{last2}, we conclude 
\begin{align}  \|\mc{C}\|_{\HS^{-1}(\mathbb{R}^{1+n})} \le C\left(\|\Lambda_1-\Lambda_2\|^{\frac{\theta}{2}}+\big|\log|\log\|\Lambda_1-\Lambda_2\||\big|^{-1}\right).\label{last3} 
\end{align}

As discussed earlier, we now translate the above norm estimates into stronger Sobolev norms using the argument of convexity. To this end, let $\alpha>1$ be arbitrary. Then there exists $\varrho\in(0,1)$ satisfying \[
\frac{1+n}{2}+\frac{\al}{2}=\varrho\left(\frac{n+1}{2}+\al\right)+(1-\varrho)(-1).
\]
Using $\|\mc{C}^{j}\|_{\HS^{\frac{n+1}{2}+\al}(\R^{1+n})}\le 4M_0$ (due to admissible set) in conjunction with Sobolev embedding and logarithmic convexity of Sobolev norms, we obtain 
\begin{align} 
\|\mc{C}^{j}\|_{\LS^\infty{(\mathbb{R}^{1+n})}}&\le C\|\mc{C}^{j}\|_{\HS^{\frac{n+1}{2}+\frac{\al}{2}}(\R^{1+n})}\\
&\le C\|\mc{C}^{j}\|_{\HS^{\frac{n+1}{2}+\al}(\R^{1+n})}^{\varrho}\|\mc{C}^{j}\|_{\HS^{-1}(\mathbb{R}^{1+n})}^{1-\varrho}\\
&\le C\left(\|\Lambda_1-\Lambda_2\|^{\frac{\theta}{2}}+\big|\log|\log\|\Lambda_1-\Lambda_2\||\big|^{-1}\right)^{1-\varrho}\\&\leq C\left(\|\Lambda_1-\Lambda_2\|^{\frac{\theta}{2}{(1-\varrho)}}+\big|\log|\log\|\Lambda_1-\Lambda_2\||\big|^{-1({1-\varrho})}\right),
\end{align}
where we have used \eqref{last3}.
Hence there exists $\mu_1,\mu_2\in(0,1)$ such that 
\begin{align}
\|\mc{C}^{j}\|_{\LS^\infty{(\mathbb{R}^{1+n})}}\le C\left(\|\Lambda_1-\Lambda_2\|^{\mu_1}+\big|\log|\log\|\Lambda_1-\Lambda_2\||\big|^{-\mu_2}\right).
\end{align}	
Taking maximum over $1\leq j\leq m$, we get
\begin{align}\label{LAST1}
\|\mc{C}\|_{\LS^\infty{(\mathbb{R}^{1+n})}}\le C\left(\|\Lambda_1-\Lambda_2\|^{\mu_1}+\big|\log|\log\|\Lambda_1-\Lambda_2\||\big|^{-\mu_2}\right).
\end{align}	
Similarly, one can get	
\begin{align}\label{bound: H1 norm of Cj}
\|\mc{C}^{j}\|_{\HS^{1}{(\mathbb{R}^{1+n})}}\le C\|\mc{C}^{j}\|_{H^{-1}(\mathbb{R}^{1+n})}^{1-\varrho}\le C\left(\|\Lambda_1-\Lambda_2\|^{\kappa_1}+\big|\log|\log\|\Lambda_1-\Lambda_2\||\big|^{-\kappa_2}\right),
\end{align}
for any $j\in\{1,2,\cdots,m\}$, and $ \kappa_1,\kappa_2\in(0,1)$.	The last estimate plays an important role during the establishment of the stability estimate for potential matrix $q$.
\subsection{Stability estimate for potential matrix \it{q}}\label{Stability estimate for potential matrix q}
This subsection is devoted to proving the stability estimate for matrix-valued potential $q$. To this end, we start by recalling the operator in
\eqref{IBVP for u (defined as differences of solution)}, that is $$\mathcal{L}_{\mathcal{A}^{(1)},\mathcal{B}^{(1)},q^{(1)}}\overrightarrow {u}(t,x) =(\mathcal{A}(t,x)\partial_t-\mathcal{B}(t,x)\cdot\nabla_x+\widetilde{\mathcal{A}}(t,x)+q(t,x))\overrightarrow{u}^{(2)}(t,x), $$
or, \begin{align}
    e^{-\lambda(t+x\cdot\og)}\mathcal{L}_{\mathcal{A}^{(1)},\mathcal{B}^{(1)},q^{(1)}}\overrightarrow {u}=   e^{-\lambda(t+x\cdot\og)} (\mathcal{A}(t,x)\partial_t-\mathcal{B}(t,x)\cdot\nabla_x+\widetilde{\mathcal{A}}(t,x)+q(t,x))\overrightarrow{u}^{(2)}.
\end{align}
Using the GO solution \eqref{GO solution for u_2 and v} and invoking \eqref{estimates for non homogeneous term for u}, we achieve
\begin{align}\label{carleman estimate for operator}
\|e^{-\lambda(t+x\cdot\og)}\mathcal{L}_{\mathcal{A}^{(1)},\mathcal{B}^{(1)},q^{(1)}}\overrightarrow {u}||_{\LS^2(Q)}^2\le  C\big(\ld^2\|\mathcal{A}\|_{\LS^\infty(Q)}^2+\ld^2\|\mathcal{B}\|_{\LS^\infty(Q)}^2+1\big)\|\overrightarrow{\phi}\|_{\HS^3(\R^n)}^2.
\end{align}
Also, for the notation convention, let us denote the right-hand side of the Carleman estimate by $\mathcal{Z}$, that is, 
\begin{align}
\mathcal{Z}:=\int_{Q}e^{-2\lambda(t+x\cdot\omega)}|\mathcal{L}_{\mathcal{A}^{(1)},\mathcal{B}^{(1)},q^{(1)}}\us(t,x)|^2dxdt
&+\int_{\Omega} e^{-2\lambda(T+x\cdot\omega)}\left(\lambda^3|\us(T,x)|^2+\ld|\nabla_x{\us(T,x)}|^2\right)dx\\&+\lambda\int\limits_{\Sigma_{-,\omega}}{ e^{-2\lambda(t+x\cdot\omega)}|\omega\cdot \nu(x)||\partial_{\nu}\us|^2}dS_x dt.   
\end{align}
Recalling \eqref{carleman estimate at time T} and \eqref{estimates on sigma-}, we observe that the second and third integrals in the previous expression have the following bounds: 
    \begin{align}
\int_{\Omega} e^{-2\lambda(T+x\cdot\omega)}\left(\lambda^3|\us(T,x)|^2+\ld|\nabla_x{\us(T,x)}|^2\right)dx \leq C e^{\beta\ld}\|\Lambda_1-\Lambda_2\|^2\|\overrightarrow{\phi}\|^2_{\HS^3(\R^n)},
\end{align}
and
\begin{align}
\begin{aligned}
\int\limits_{\Sigma_{-,\omega}}\lambda e^{-2\lambda(t+x\cdot\omega)}|\omega\cdot \nu(x)||\partial_{\nu}\us|^2dS_xdt\leq C e^{\beta\ld}\|\Lambda_1-\Lambda_2\|^2\|\overrightarrow{\phi}\|^2_{\HS^3(\R^n)}.
\end{aligned}
\end{align}
The above bounds together with \eqref{carleman estimate for operator} lead to   
\begin{align}\label{Revised estimate for RHS of carleman estimate}
\mc{Z}\le C\big(\ld^2\|\mathcal{A}\|_{\LS^\infty(Q)}^2+\ld^2\|\mathcal{B}\|_{\LS^\infty(Q)}^2+1+ e^{\beta\lambda}\|\Lambda_1-\Lambda_2\|^2\big)\|\overrightarrow{\phi}\|_{\HS^3(\mathbb{R}^n)}^2. 
\end{align}
Using \eqref{Revised estimate for RHS of carleman estimate}, we modify the estimate \eqref{final estimates with all matrix potential} as follows
\begin{align}
   \Bigg|\int\limits_Q \big(\mathcal{A}(t,x)\partial_t\overrightarrow{u}^{(2)}(t,x)-\mathcal{B}(t,x)\cdot\nabla_x\overrightarrow{u}^{(2)}(t,x)+(\widetilde{\mathcal{A}}(t,x)+q(t,x))\overrightarrow{u}^{(2)}(t,x)\big)\cdot\overline{\overrightarrow{v}(t,x)}\,dxdt\Bigg|\\
\leq C\|\overrightarrow{\phi}\|^2_{\HS^3(\mathbb{R}^n)}\left(\sqrt{\frac{\mc{Z}}{\lambda}} +e^{\beta\lambda}\|\Lambda_1-\Lambda_2\|\right).
\end{align}
On further simplification, we obtain
\begin{align}
    \bigg|\int_{Q}(\widetilde{\mathcal{A}}(t,x)+q(t,x))&\overrightarrow{u}^{(2)}(t,x)\big)\cdot\overline{\overrightarrow{v}(t,x)}\,dxdt\bigg|\\
&\le C\left(\ld\|\mathcal{A}\|_{\LS^\infty(Q)}+\ld\|\mathcal{B}\|_{\LS^\infty(Q)}+\frac{1}{\rt{\ld}}+ e^{\beta\lambda}\|\Lambda_1-\Lambda_2\|\right)\|\overrightarrow{\phi}\|_{\HS^3(\mathbb{R}^n)}^2,\nn\\
\end{align}
or,
\begin{align}
&\Big|\int_{Q}\sum_{\ell=1}^m\Big(\widetilde{A^\ell}(t,x)u^{(2)}_\ell(t,x)\overline{v_\ell(t,x)}\,dxdt+\sum_{k=1}^m q_{\ell k}(t,x)u^{(2)}_k(t,x)\overline{v_\ell(t,x)}\Big)dxdt\Big|\\
&\qquad\qquad\qquad\le C\left(\ld\|\mathcal{A}\|_{\LS^\infty(Q)}+\ld\|\mathcal{B}\|_{\LS^\infty(Q)}+\frac{1}{\rt{\ld}}+ e^{\beta\lambda}\|\Lambda_1-\Lambda_2\|\right)\|\overrightarrow{\phi}\|_{\HS^3(\mathbb{R}^n)}^2.
\end{align}
Substituting the GO solution \eqref{GO solution for u_2 and v} into the previous expression to get
\begin{align}\label{R1}
\begin{aligned}
&\Big|\int\limits_{Q}\sum_{\ell=1}^m \Big(\widetilde{A^\ell}(t,x)K^\ell K_*^{\ell}|\phi_\ell(x+t\omega)|^2e^{-\int_0^t (1,-\omega)\cdot\mathcal{C}^{\ell}(s,x+(t-s)\omega)\,ds}+\sum_{k=1}^m q_{\ell k}u^{(2)}_k(t,x)\overline{v_\ell(t,x)}\Big)dxdt\Big|\\
&\qquad\qquad\qquad\le C\left(\ld\|\mathcal{A}\|_{\LS^\infty(Q)}+\ld\|\mathcal{B}\|_{\LS^\infty(Q)}+\frac{1}{\rt{\ld}}+ e^{\beta\lambda}\|\Lambda_1-\Lambda_2\|\right)\|\overrightarrow{\phi}\|_{\HS^3(\mathbb{R}^n)}^2.
\end{aligned}
\end{align}
\subsubsection{Case: $i\ne j$}
We start by choosing constant vectors $\overrightarrow {K}$ and $\overrightarrow {K_*}$ introduced in \eqref{jth component of GO solution for u2} and \eqref{jth component of GO solution for v}, respectively.
For arbitrary but fixed $i,j\in\{1,2,\cdots,m\}$, we choose $\overrightarrow {K}$ and $\overrightarrow {K_*}\in \mathbb{R}^m$ as follows:
\begin{equation*}
    \overrightarrow {K}=\left( \delta_{jr} \right)^m_{r=1},\,\,\quad \overrightarrow {K_*}=\left( \delta_{ir}\right)^m_{r=1},
\end{equation*}
where $\delta_{kr}$ represents the Kronecker delta function for non-negative integers $k$ and $r$. Therefore, we have
\begin{align*}
&\left|\int_{Q}q_{ij}(t,x)u^{(2)}_j(t,x)\overline{v_i(t,x)}\Big)dxdt\right|\\&\qquad=\left|\int_{Q}q_{ij}(t,x)\phi_j(x+t\omega)\phi_i(x+t\omega)e^{-\int_0^t (1,-\omega)\cdot(\mathcal{C}^{(2),j}-\mathcal{C}^{(1),i})(s,x+(t-s)\omega)\,ds}dxdt\right|\\
&\qquad\le C\left(\ld\|\mathcal{A}\|_{\LS^\infty(Q)}+\ld\|\mathcal{B}\|_{\LS^\infty(Q)}+\frac{1}{\rt{\ld}}+ e^{\beta\lambda}\|\Lambda_1-\Lambda_2\|\right)\|\overrightarrow{\phi}\|_{\HS^3(\mathbb{R}^n)}^2.
\end{align*} 
Further, we observe that
\begin{align*}  
&\left|\int_{Q}q_{ij}(t,x)\phi_j(x+t\omega)\phi_i(x+t\omega)\ dxdt\right|\\
&\qquad\le \left|\int_{Q}q_{ij}(t,x)\phi_j(x+t\omega)\phi_i(x+t\omega)\left(e^{-\int_0^t (1,-\omega)\cdot(\mathcal{C}^{(2),j}-\mathcal{C}^{(1),i})(s,x+(t-s)\omega)\,ds}-1\right)dxdt\right|\\
& \qquad\qquad +\left|\int_{Q}q_{ij}(t,x)\phi_j(x+t\omega)\phi_i(x+t\omega)e^{-\int_0^t (1,-\omega)\cdot(\mathcal{C}^{(2),j}-\mathcal{C}^{(1),i})(s,x+(t-s)\omega)\,ds}dxdt\right|.
\end{align*}
To simplify the first term on the right-hand side of the previous expression, we use the 
mean value theorem and a priori bounds of vector potentials. Therefore, using \eqref{Def: extra condition}, we get
\begin{align}\label{M.V.T}
\begin{aligned}
&\big|e^{-\int_{0}^{t}(1,-\og)\cdot(\mathcal{C}^{(2),j}-\mathcal{C}^{(1),i})(s,x+(t-s)\omega)\,ds}-1\big|\le C\|\mathcal{C}^{(2),j}-\mathcal{C}^{(1),i}\pm\mathcal{C}^{(2),i}\|_{L^{\infty}(Q)}\le C \lVert \mathcal{C}\rVert_{\LS^{\infty}(Q)}.
\end{aligned}
\end{align}
Thus, we have
\begin{align*}  
&\left|\int_{Q}q_{ij}(t,x)\phi_j(x+t\omega)\phi_i(x+t\omega)\ dxdt\right|\\
&\qquad\le C\left(\ld\|\mathcal{A}\|_{\LS^\infty(Q)}+\ld\|\mathcal{B}\|_{\LS^\infty(Q)}+\frac{1}{\rt{\ld}}+ e^{\beta\lambda}\|\Lambda_1-\Lambda_2\|\right)\|\overrightarrow{\phi}\|_{\HS^3(\mathbb{R}^n)}^2.
\end{align*}
We set $\phi_i=\phi_j$, and perform the change of variables $y=x+t\omega$, which yields
\begin{align*}
&\left|\int_{0}^T\int_{\mathbb{R}^n}q_{ij}(t,y-t\omega)|\phi_i(y)|^2 \ dydt\right|\\
&\qquad\le C\left(\ld\|\mathcal{A}\|_{\LS^\infty(Q)}+\ld\|\mathcal{B}\|_{\LS^\infty(Q)}+\frac{1}{\rt{\ld}}+ e^{\beta\lambda}\|\Lambda_1-\Lambda_2\|\right)\|\overrightarrow{\phi}\|_{\HS^3(\mathbb{R}^n)}^2.
\end{align*}
By arguments analogous to those employed in Lemma \ref{existence of delta lambda0} and Corollary \ref{existence of a cone}, there exist $\gamma,\delta>0$ and an open cone $\mathcal{S}$ such that
\begin{align}\label{R} 
\begin{aligned}
\Big|\int_{0}^{T}q_{ij}(t,y-t\omega)dt\Big|& \m{ and,}\ \|\widehat{q_{ij}}\|_{L^\infty(\mc{S})}\\&\le C\left(\ld^{\gamma}\big(\|\mc{A}\|_{\LS^\infty(Q)}+\|\mathcal{B}\|_{\LS^\infty(Q)}\big)+\frac{1}{\ld^{\delta}}+ e^{\beta\lambda}\|\Lambda_1-\Lambda_2\|\right).
\end{aligned}
\end{align}
Moreover, note that for $\widetilde{q} := q + \widetilde{\mathcal{A}}$ (see notational convention \eqref{notation used in integral identity}), the equality $\widetilde{q} = q$ holds whenever $i \neq j$. Consequently, the preceding expression can be reformulated as
\begin{align}
\begin{aligned}
\Big|\int_{0}^{T}\widetilde{q_{ij}}(t,y-t\omega)dt\Big|& \m{ and,}\ \|\widehat{\widetilde{q_{ij}}}\|_{L^\infty(\mc{S})}\\&\le C\left(\ld^{\gamma}\big(\|\mc{A}\|_{\LS^\infty(Q)}+\|\mathcal{B}\|_{\LS^\infty(Q)}\big)+\frac{1}{\ld^{\delta}}+ e^{\beta\lambda}\|\Lambda_1-\Lambda_2\|\right).
\end{aligned}
\end{align}

\subsubsection{Case: $i=j$}
Let us denote $\widetilde{q}:= q+\widetilde{\mathcal{A}}$ and choose $ \overrightarrow {K}=\left( \delta_{jr} \right)^m_{r=1},~\overrightarrow {K_*}=\left( \delta_{jr}\right)^m_{r=1}$ in \eqref{R1}
then 
\begin{align}
\begin{aligned}
\Big|\int_{0}^{T}\widetilde{q_{ii}}(t,y-t\omega)dt\Big|& \m{ and,}\ \|\widehat{\widetilde{q_{ii}}}\|_{L^\infty(\mc{S})}\\&\le C\left(\ld^{\gamma}\big(\|\mc{A}\|_{\LS^\infty(Q)}+\|\mathcal{B}\|_{\LS^\infty(Q)}\big)+\frac{1}{\ld^{\delta}}+ e^{\beta\lambda}\|\Lambda_1-\Lambda_2\|\right).
\end{aligned}
\end{align}
where $\widetilde{q_{ii}}$ denotes the $(i,i)^{th}$ entry of $\widetilde{q}$. Also, we have
\begin{align}\label{bound: q in term of tilde q}
    \lVert q_{ij}\rVert_{L^{\infty}(\mathbb{R}^{1+n})}\leq\begin{cases}
        C\left(\lVert \widetilde{q_{ii}}\rVert_{L^{\infty}(\mathbb{R}^{1+n})}+\lVert \mathcal{C}^i\rVert_{H^1(\mathbb{R}^{1+n})}+\lVert \mathcal{C}^i\rVert_{L^{\infty}(\mathbb{R}^{1+n})}, 
    \right), & i= j,\\  C\lVert \widetilde{q_{ij}}\rVert_{L^{\infty}(\mathbb{R}^{1+n})}, & i\neq j.
    \end{cases} 
\end{align}
We first analyze $\widetilde{q_{ii}}$, subsequently, by invoking the relation \eqref{bound: q in term of tilde q}, we deduce the desired estimate.
Using the Vessella's analytic continuation result as done in \eqref{vassella2}, we obtain
\begin{align}
&\|\widehat{\widetilde{q_{ij}}}\|_{L^\infty(B(0,k))}\ \le Ce^{k(1-\theta)}\left(\ld^{\gamma}\big(\|\mc{A}\|_{L^\infty(Q)}+\|\mathcal{B}\|_{L^\infty(Q)}\big)+\frac{1}{\ld^{\delta}}+ e^{\beta\lambda}\|\Lambda_1-\Lambda_2\|\right)^{\theta},
\end{align}
for large $k$ and $1\leq i,j\leq m$.
Then, we have
\begin{align}\label{sc1}
\begin{aligned}
&\|\widetilde{q_{ij}}\|_{H^{-1}{(\R^{1+n})}}^{\frac{2}{\theta}}\\&\qquad
\le C\left(k^{\frac{n+1}{\theta}}e^{\frac{2k(1-\theta)}{\theta}}\left(\ld^{2\gamma}\big(\|\mc{A}\|^2_{L^\infty(Q)}+\|\mathcal{B}\|^2_{L^\infty(Q)}\big)+\frac{1}{\ld^{2\delta}}+e^{\beta\ld}\|\Lambda_1-\Lambda_2\|^2\right)+k^{-\frac{2}{\theta}}\right)
\end{aligned}
\end{align}
for large $k$.
 Repeating a similar argument as in $\eqref{last1}$, we choose $\ld=e^{\frac{k(1-\theta)}{\theta\delta}}k^{\frac{n+3}{2\theta\delta}}$. Next, we use the estimate \eqref{LAST1} in \eqref{sc1}, then there exists $\varrho>0$ such the right-hand side of \eqref{sc1} can be bounded by
\begin{align}
e^{\varrho k}\|\Lambda_1-\Lambda_2\|^{2\mu_1}+e^{\varrho k}\big|\log|\log\|\Lambda_1-\Lambda_2\||\big|^{-2\mu_2}+e^{e^{\varrho k}}\|\Lambda_1-\Lambda_2\|^2+k^{-\frac{2}{\theta}}.
\end{align}
Now, if we choose $k=\frac{\mu_2}{\varrho}\log\log\big|\log\|\Lambda_1-\Lambda_2\|\big|$ (provided $\|\Lambda_1-\Lambda_2\|$ is sufficiently small), the foregoing expression reduces to    
\begin{align}\label{sc2} 
\begin{aligned}
\|\Lambda_1-\Lambda_2\|^{2\mu_1}\big|\log|\log\|\Lambda_1-\Lambda_2\||\big|^{\mu_2}+{\big|\log|\log\|\Lambda_1-\Lambda_2\||\big|^{-\mu_2}}\\
+{{\|\Lambda_1-\Lambda_2\|^{2}\left|\log\|\Lambda_1-\Lambda_2\|\right|}}+\big|\log|\log\|\Lambda_1-\Lambda_2\||\big|^{-\frac{2}{\theta}}.
\end{aligned}
\end{align}
From above, we observe that the condition $\|\Lambda_1-\Lambda_2\|$ is sufficiently small, by itself, is inadequate to derive the desired bound. Consequently, we impose an additional assumption. Specifically, we assume that $\|\Lambda_1-\Lambda_2\|$ is small enough- say, that $\|\Lambda_1-\Lambda_2\|<c$ so that the choice of $k$ is valid and the following holds:
\begin{enumerate}
    \item $\|\Lambda_1-\Lambda_2\|^{2\mu_1}\le C(\log|\log\|\Lambda_1-\Lambda_2\||)^{-2\mu_2}.$
    \item $\left(\log|\log\|\Lambda_1-\Lambda_2\||\right)^{-\mu_2}\le (\log|\log\left|\log\|\Lambda_1-\Lambda_2\|\right||)^{-\mu_2}$.
    \item $\|\Lambda_1-\Lambda_2\|\left|\log\|\Lambda_1-\Lambda_2\|\right|\le C$.
\end{enumerate}
With these assumption,
we may dominate \eqref{sc2} by the following term \begin{align}\|\Lambda_1-\Lambda_2\|^{a_1}+\left(\log|\log\left|\log\|\Lambda_1-\Lambda_2\|\right||\right)^{-a_2}\m{ for some }a_1,a_2>0.\label{scalar1}
\end{align} 
For $\|\Lambda_1-\Lambda_2\|\ge c$, using the argument as in \eqref{last2}, we obtain a similar estimates as \eqref{scalar1}. Finally, invoking Sobolev embedding and logarithmic convexity of Sobolev norms, we conclude 
\begin{align*}
\|\widetilde{q_{ij}}\|_{L^\infty{(\R^{1+n})}}\le C\left(\|\Lambda_1-\Lambda_2\|^{\kappa_3}+\Big|\log\big|\log|\log\|\Lambda_1-\Lambda_2\||\big|\Big|^{-\kappa_4}\right)\quad \ \m{for some } \kappa_3,\kappa_4 >0,
\end{align*}
where $i,j\in\{1,2,3,\dots,m\}$.
The use of \eqref{LAST1}, \eqref{bound: H1 norm of Cj} and \eqref{bound: q in term of tilde q} in the above expression, yields 
\begin{align*}
\|q_{ij}\|_{L^\infty{(\R^{1+n})}}\le C\left(\|\Lambda_1-\Lambda_2\|^{\mu_3}+\Big|\log\big|\log|\log\|\Lambda_1-\Lambda_2\||\big|\Big|^{-\mu_4}\right)\quad \ \m{for some } \mu_3,\mu_4 >0,
\end{align*}
where $1\leq i,j\leq m$.
Taking maximum over $1\leq i,j\leq m$, we have 
\begin{align*}
\|q\|_{\LS^\infty{(\R^{1+n})}}\le C\left(\|\Lambda_1-\Lambda_2\|^{\mu_3}+\Big|\log\big|\log|\log\|\Lambda_1-\Lambda_2\||\big|\Big|^{-\mu_4}\right)\quad \ \m{for some } \mu_3,\mu_4 >0.
\end{align*}
This completes the proof.
 \section*{Acknowledgments}
 \noindent 
  Parveen Kumar is pleased to acknowledge CSIR, India, for the financial support provided through the Senior Research Fellowship (File No. 09/1005(19269)/2024-EMR-I). Parveen Kumar also thanks the FIST Programme of the Department of Science and Technology (DST), Government of India (Reference No. SR/FST/MS-I/2018/22(C)) for partial financial support. 
 The authors sincerely thank Manmohan Vashisth for his insightful remarks and valuable suggestions.
\bibliography{references}
\bibliographystyle{amsplain}
\end{document}